\documentclass[reqno]{amsart}
\usepackage{amssymb,amscd,graphicx,stmaryrd,mathrsfs,bbm,accents}
\usepackage[all]{xy}
\usepackage{color}
\usepackage{url}
\usepackage%[dvipdfmx]
{hyperref}

\newtheorem{theorem}{Theorem}[section]
\newtheorem{proposition}[theorem]{Proposition} 

\newtheorem{lemma}[theorem]{Lemma}
\newtheorem{remark}[theorem]{Remark}

\newtheorem{remark&definition}[theorem]{Remark and Definition}

\newtheorem{definition}[theorem]{Definition}

\makeatletter
  
  \@addtoreset{equation}{section}
\makeatother

\begin{document}

\title[Spherical representations for $C^*$-flows III]{Spherical representations for $C^*$-flows III: \\
\small Weight-extended branching graphs} 
\author{Yoshimichi UEDA}
\address{
Graduate School of Mathematics, Nagoya University, 
Furocho, Chikusaku, Nagoya, 464-8602, Japan
}
\email{ueda@math.nagoya-u.ac.jp}
\date{\today}
\thanks{This work was supported by Grant-in-Aid for Scientific Research (B) JP18H01122.}
\maketitle
\begin{abstract}
We apply Takesaki's and Connes's ideas on structure analysis for type III factors to the study of links (a short term of Markov kernels) appearing in asymptotic representation theory. 
\end{abstract}

\allowdisplaybreaks{

\tableofcontents

\section{Introduction}\label{S1}

Asymptotic representation theory was initiated by Vershik and Kerov around 1980, and investigates unitary characters of inductive limits of finite/compact groups. The theory involves several operator algebraic tools such as AF-algebras with their dimension groups since its birth; see e.g., \cite{Kerov:Book}. The main classification problem (on factor representations) in the theory is described in terms of links (or equivalently, Markov kernels) on branching graphs; see e.g., \cite{Kerov:Book},\cite {BorodinOlshanski:Book}. For the infinite symmetric group (the inductive limit of symmetric groups), the branching graph is Young poset and the link is obtained from its multiplicity function. In this way, the study of asymptotic representation theory for \emph{ordinary} groups can be studied by looking at only branching graphs. However, one can consider links that do not match multiplicity functions. Such a link naturally arises in the quantum group setting as an effect of $q$-deformation (see \cite{Gorin:AdvMath12},\cite{Sato:JFA19}), and we have developed, in \cite{Ueda:preprint20,Ueda:SIGMA22}, an abstract framework to discuss those from the viewpoint of Olshanski's spherical representation theory in the general operator algebraic setting. This paper is a kind of supplements to those previous papers, and the purpose is to introduce a new method of studying general links on branching graphs, which admits a $K$-theoretic interpretation. 

Our abstract framework naturally leads us to the use of Takesaki's idea \cite{Takesaki:ActaMath73} on general structure analysis for type III factors (based on his celebrated duality theorem) and Connes's idea \cite{Connes:JFA74} on almost-periodic weights in the study of links that do not match multiplicity functions. Namely, the framework starts with an inductive sequence $A_n$ of atomic $W^*$-algebras with continuous flows $\alpha_n^t : \mathbb{R} \curvearrowright A_n$, and then take its inductive limit $(A,\alpha^t) = \varinjlim (A_n,\alpha_n^t)$. We apply the construction of Takesaki duals to the inductive sequence and obtain a new inductive sequence $\widetilde{A}_n$ of atomic $W^*$-algebras again equipped with actions $\widetilde{\alpha}_n^\gamma$ of discrete subgroup $\Gamma$ of the multiplicative group $\mathbb{R}_+^\times$. Take the inductive limit $(\widetilde{A},\widetilde{\alpha}^\gamma) = \varinjlim (\widetilde{A}_n,\widetilde{\alpha}_n^\gamma)$, and the locally normal $(\alpha^t,\beta)$-KMS states $K_\beta^\mathrm{ln}(\alpha^t)$ are shown to be affine-isomorphic to the tracial weights $\tau$ on $\widetilde{A}$ that are locally normal semifinite and suitably scaling under $\widetilde{\alpha}^\gamma$. This procedure is explained in section \ref{S3}. We then interpret this procedure in terms of links on branching graphs. This is done in section \ref{S4}. A consequence is that the study of a general link on a branching graph is reduced to that of the link arising from the multiplicity function on an extended branching graph \emph{with group action}. This new approach allows us to use the notion of dimension groups explicitly. The reader who is only interested in the study of links may directly go to section \ref{S4.3}, where the present method is given without appealing to any operator algebras. In section \ref{S5}, we examine a relation between the present method and $K$-theory. A consequence is to give a way to connect the study of general links to $K_0$-groups. In final section \ref{S6}, we examine the present method with the infinite dimensional quantum unitary group $\mathrm{U}_q(\infty)$, whose formulation was precisely given in part II of this series of papers. The consequence there explains that the present method is closed in the class of inductive limits of compact quantum groups and should be regarded as a way to make the special positive elements $\rho_n \in \mathcal{U}(\mathrm{U}_q(n))$, $n=0,1,\dots$, (see, e.g., \cite[equation (4.5)]{Ueda:SIGMA22}) form an inductive sequence by enlarging algebras in question. See section \ref{S6.2}.

\medskip
We will use the following notation rule: $\mathcal{F} \Subset \Gamma$ means that $\mathcal{F}$ is a finite subset of a set $\Gamma$. For a $C^*$-algebra $C$, we denote by $C_+$ the cone of its positive elements. We also mention that our main references on operator algebras are still Bratteli--Robinson's books \cite{BratteliRobinson:Book1,BratteliRobinson:Book2} as our previous two papers \cite{Ueda:preprint20,Ueda:SIGMA22}, but we have to refer to Takesaki's book vol.II \cite{Takesaki:Book2} concerning weights on $C^*$-/$W^*$-algebras and the so-called Tomita--Takesaki theory with its applications to type III factors.   

\section{General setup}\label{S2}

Let $A_n$, $n=1,2,\dots$, be atomic $W^*$-algebras with separable preduals, and put $A_0=\mathbb{C}1$. We assume that the $A_n$ form an inductive sequence by unital normal embeddings $A_n \hookrightarrow A_{n+1}$, $n=0,1,\dots$. Let $A = \varinjlim A_n$ be the inductive (direct) limit $C^*$-algebra. For each $n$, we denote by $\mathfrak{Z}_n$ all the minimal projections in $\mathcal{Z}(A_n)$. 

Assume that we have a flow $\alpha : \mathbb{R} \curvearrowright A$ such that $\alpha^t(A_n) = A_n$ holds for every $t\in\mathbb{R}$ and $n\geq0$ (that is, $\alpha^t$ is an \emph{inductive flow}) and moreover that the restriction of $\alpha^t$ to each $A_n$, denoted by $\alpha_n^t : \mathbb{R} \curvearrowright A_n$, is continuous in the $u$-topology. The $u$-continuity assumption makes every flow $\alpha_n^t$ fix elements in $\mathcal{Z}(A_n)$. See \cite[Lemma 7.1]{Ueda:preprint20} for details. Thus, for each $z \in \mathfrak{Z}_n$, $n \geq 0$, the restriction of $\alpha_n^t$ to $zA_n$ defines a `local' flow $\alpha_z^t$. 

For each $z\in\mathfrak{Z}_n$, $zA_n$ is identified with all the bounded operators $B(\mathcal{H}_z)$ on a Hilbert space $\mathcal{H}_z$, since $A_n$ is atomic. Then, for each $z \in \mathfrak{Z}_n$, $n \geq 0$, we can find a unique (up to positive scaling) non-singular positive self-adjoint operator $\rho_z$ affiliated with $zA_n = B(\mathcal{H}_z)$ such that $\alpha_z^t = \mathrm{Ad}\rho_z^{it}$ for every $t \in \mathbb{R}$. \emph{Throughout this paper, we will consider only the case when all $\rho_z$ are diagonalizable.} This is fulfilled when all the \emph{dimensions} $\dim(z) := \dim(\mathcal{H}_z)<\infty$.    

To the inductive sequence $A_n$ we associate a branching graph together with multiplicity function as follows. The vertex set is $\mathfrak{Z} = \bigsqcup_{n\geq0}\mathfrak{Z}_n$, and the multiplicity function $m : \bigsqcup_{n\geq0} \mathfrak{Z}_{n+1}\times\mathfrak{Z}_n \to \mathbb{N}\cup\{0,\infty\}$ is defined to be the multiplicity of $z' A_n = B(\mathcal{H}_{z'})$ in $zA_{n+1} = B(\mathcal{H}_z)$ via $A_n \hookrightarrow A_{n+1}$ for $(z,z') \in \mathfrak{Z}_{n+1}\times\mathfrak{Z}_n$. We observe that 
\[
\bigcup_{z'\in\mathfrak{Z_n}} \{z \in \mathfrak{Z}_{n+1}; m(z,z') > 0\} = \mathfrak{Z}_{n+1}, \qquad \bigcup_{z\in\mathfrak{Z}_{n+1}} \{z' \in \mathfrak{Z}_n ; m(z,z') > 0\} = \mathfrak{Z}_n
\] 
for all $n\geq0$, and  
\begin{equation}\label{Eq2.1}
\mathrm{Tr}(zz') = m(z,z')\,\dim(z'), \qquad (z,z') \in  \bigsqcup_{n\geq0} \mathfrak{Z}_{n+1}\times\mathfrak{Z}_n  
\end{equation}
hold, where $\mathrm{Tr}$ stands for the non-normalized trace on $zA_{n+1} = B(\mathcal{H}_z)$. We also remark that 
\[
\dim(z) = \sum_{z' \in \mathfrak{Z}_{n-1}} m(z,z') \dim(z') = \cdots = \sum_{z_i \in \mathfrak{Z}_i (i=1,\dots,n-1)} m(z,z_{n-1})\cdots m(z_2,z_1) m(z_1,1)
\] 
for every $z \in \mathfrak{Z}_n$. The edge set is defined to be all the $(z,z') \in \bigsqcup_{n\geq0} \mathfrak{Z}_{n+1}\times\mathfrak{Z}_n$ with $m(z,z') > 0$. We have known (see \cite[section 9]{Ueda:preprint20}) that the graph $(\mathfrak{Z},m)$ completely remembers the inductive sequence $A_n$. 

Let an inverse temperature $\beta \in \mathbb{R}$ be fixed throughout in such a way that $\mathrm{Tr}(\rho_z^{-\beta}) < \infty$ for all $z \in \mathfrak{Z}_n$, $n\geq1$. For each $z \in \mathfrak{Z}_n$, $n\geq0$, a unique (faithful, normal) $(\alpha_z^t,\beta)$-KMS state $\tau_z^\beta = \tau^{(\alpha_z^t,\beta)}$ on $zA_n = B(\mathcal{H}_z)$ is given by 
\begin{equation}\label{Eq2.2}
x \in B(\mathcal{H}_z) \mapsto \tau_z^\beta(x) := \frac{\mathrm{Tr}(\rho_z^{-\beta} x)}{\mathrm{Tr}(\rho_z^{-\beta})} \in \mathbb{C}.
\end{equation} 
In what follows, we write $\dim_\beta(z) = \dim_{(\alpha^t,\beta)}(z) := \mathrm{Tr}(\rho_z^{-\beta})$. 

We have discussed, in \cite{Ueda:preprint20,Ueda:SIGMA22}, locally normal $(\alpha^t,\beta)$-spherical representations for $A = \varinjlim A_n$, whose classification can be understood in terms of link over $\mathfrak{Z} = \bigsqcup_{n\geq0}\mathfrak{Z}_n$. In the present setting, the link $\kappa = \kappa_{(\alpha^t,\beta)} : \bigsqcup_{n\geq0} \mathfrak{Z}_{n+1}\times\mathfrak{Z}_n \to [0,1]$ is given by 
\begin{equation}\label{Eq2.3}
\kappa(z,z') := \tau_z^\beta(zz') = \frac{\mathrm{Tr}(\rho_z^{-\beta}z')}{\dim_\beta(z)}, \qquad (z,z') \in  \bigsqcup_{n\geq0} \mathfrak{Z}_{n+1}\times\mathfrak{Z}_n. 
\end{equation}
If $\beta=0$ and all $\dim(z) < \infty$, then $\dim_\beta(z) = \dim(z)$ holds for every $z \in \mathfrak{Z}$ and the link $\kappa(z,z')$ is nothing less than 
\[
\mu(z,z') := \frac{1}{\dim(z)}\,m(z,z')\,\dim(z'), \qquad (z,z') \in  \bigsqcup_{n\geq0} \mathfrak{Z}_{n+1}\times\mathfrak{Z}_n.
\] 
We call this special link $\mu : \bigsqcup_{n\geq0} \mathfrak{Z}_{n+1}\times\mathfrak{Z}_n \to [0,1]$ the \emph{standard link} (this is available only when all $\dim(z)<\infty$). The standard link fits the notion of dimension groups, but other links do not. 
Consequently, to a given branching graph $(\mathfrak{Z},m)$ we associate the standard link $\mu$ under all the $\dim(z)<\infty$, but a non-standard link on $(\mathfrak{Z},m)$ can also be considered even when $\mu$ cannot. Moreover, we illustrated in \cite[section 9]{Ueda:preprint20} how any non-standard link arises in the spherical representation theory for a certain class of $C^*$-flows. 

\section{$\rho$-extension}\label{S3}

We fix a family $\rho = \{\rho_z\}_{z\in\mathfrak{Z}}$ as in section \ref{S2}, that is, each $\rho_z^{it}$ implements the restriction $\alpha_z^t$ of $\alpha_n^t$ to $zA_n$, $z \in \mathfrak{Z}_n \subset \mathfrak{Z}$, and all $\rho_z$ are diagonalizable. Let $\Gamma=\Gamma(\rho)$ be the discrete (countable) subgroup generated by all the eigenvalues of $\rho_z$'s in the multiplicative group $\mathbb{R}_+^\times = (0,\infty)$. Let $G = \widehat{\Gamma}$ be the dual compact abelian group of $\Gamma$.  There is a continuous homomorphism from $\mathbb{R}$ into $G$ with dense image such that $\langle \gamma,t\rangle = \gamma^{it}$ holds for every $\gamma \in \Gamma$ when $t \in \mathbb{R}$ is regarded as an element of $G$ via the homomorphism, where $\langle\,\cdot\,,\,\cdot\,\rangle : \Gamma \times G \to \mathbb{T}$ is the dual pairing. It is evident that every unitary representation $t\mapsto u_z(t) = \rho_z^{it}$ of the real numbers $\mathbb{R}$ uniquely extends to $G$ by using the spectral decomposition of $\rho_z$, and hence so does every flow $\alpha_n^t$ too. 

\medskip
For each $n=0,1,\dots$, we take the $W^*$-crossed product $\widetilde{A}_n := A_n\bar{\rtimes}_{\alpha_n^g} G$, whose construction (see e.g., \cite[Definition 2.7.3]{BratteliRobinson:Book1}) is reviewed in our convenient way as follows. Since $A_n$ has separable predual and thus is $\sigma$-finite, $A_n$ acts on a Hilbert space $\mathcal{K}_n$ with separating and cyclic vector. (See e.g., \cite[Proposition 2.5.6]{BratteliRobinson:Book1}.) Let $L^2(G;\mathcal{K}_n)$ be the $\mathcal{K}_n$-valued $L^2$-space over $G$ with respect to the Haar probability measure $dg$, which can be identified with the completion of the $\mathcal{K}_n$-valued continuous functions $C(G;\mathcal{K}_n)$ equipped with inner product 
\[
(\xi\,|\,\eta) := \int_G (\xi(g)\,|\,\eta(g))_{\mathcal{K}_n}\,dg, \qquad \xi, \eta \in C(G;\mathcal{K}_n). 
\]
We define an injective normal $*$-homomorphism $\pi_{\alpha_n} : A_n \to B(L^2(G;\mathcal{K}_n))$ by 
\[
(\pi_{\alpha_n}(a)\xi)(g) := \alpha_n^{g^{-1}}(a)\xi(g), \qquad a \in A_n, \quad \xi \in C(G;\mathcal{K}_n) \subset L^2(G;\mathcal{K}_n).
\] 
Let $\lambda : G \curvearrowright L^2(G;\mathcal{K}_n)$ be the unitary representation defined by 
\[
(\lambda(g_1)\xi)(g_2):=\xi(g_1^{-1}g_2), \qquad g_1,g_2 \in G, \quad \xi \in C(G;\mathcal{K}_n) \subset L^2(G;\mathcal{K}_n).
\] 
We have a natural identification $L^2(G;\mathcal{K}_n) = \mathcal{K}_n\,\bar{\otimes}\,L^2(G)$ by 
\[
(\xi\otimes f)(g) = f(g)\xi, \qquad \xi \in \mathcal{K}_n, \quad f \in C(G) \subset L^2(G), 
\]
where $C(G) \subset L^2(G)$ denote the continuous functions on $G$ and the $L^2$-space over $G$ with respect to $dg$, respectively. Via the identification we have 
\[
\lambda(g) := 1\otimes\lambda_g, \qquad g \in G
\] 
with the left regular representation $\lambda_g$ of $G$. Then, the $W^*$-crossed product $A_n\bar{\rtimes}_{\alpha_n^g} G$ is the $W^*$-subalgebra of $A_n\,\bar{\otimes}\,B(L^2(G))$ generated by $\pi_{\alpha_n}(A_n)$ and $\lambda(G)$ in $A_n\,\bar{\otimes}\,B(L^2(G))$ with covariant relation 
\[
\lambda(g)\pi_{\alpha_n}(a) = \pi_{\alpha_n}(\alpha_n^g(a))\lambda(g), \qquad a \in A_n, \quad g \in G. 
\]
Remark that (the algebraic structure of) the resulting $W^*$-algebra $A_n\bar{\rtimes}_{\alpha_n^g} G$ is known to be independent of the choice of representation $A_n \subset B(\mathcal{K}_n)$; see \cite[section X.1]{Takesaki:Book2}. 

\medskip
We observe that $\widetilde{A}_0 = \mathbb{C}1\bar{\rtimes}G \cong  \ell^\infty(\Gamma)$ is given by 
\begin{equation}\label{Eq3.1}
e_\gamma = \int_G \overline{\langle \gamma, g\rangle}\,\lambda(g)\,dg \longleftrightarrow \delta_\gamma,   
\end{equation}
where $\delta_\gamma$ is the Dirac function at $\gamma$. The so-called \emph{dual action} $\widetilde{\alpha}_n : \Gamma \curvearrowright \widetilde{A}_n$ (see e.g., \cite[Definition 2.7.3]{BratteliRobinson:Book1}) can be constructed in such a way that 
\begin{equation}\label{Eq3.2} 
\widetilde{\alpha}_n^\gamma(\pi_{\alpha_n}(a)) = \pi_{\alpha_n}(a), \quad \widetilde{\alpha}_n^\gamma(\lambda(g)) = \overline{\langle \gamma,g\rangle}\,\lambda(g) \qquad a \in A_n, \quad \gamma \in \Gamma, \quad g \in G, 
\end{equation} 
and the latter relation is rephrased as 
\begin{equation}\label{Eq3.3}
\widetilde{\alpha}_n^\gamma(e_{\gamma'}) = e_{\gamma\gamma'},  \qquad \gamma,\gamma' \in \Gamma. 
\end{equation} 

Since $\alpha_{n+1}^g = \alpha_n^g$ holds on $A_n$ for every $g \in G$, we have a normal embedding $\widetilde{A}_n \hookrightarrow \widetilde{A}_{n+1}$ determined by 
\begin{equation}\label{Eq3.4}
\pi_{\alpha_n}(a) \mapsto \pi_{\alpha_{n+1}}(a), \qquad a \in A_n. 
\end{equation}   
Hence, the $\widetilde{A}_n$ form an inductive sequence, and let $\widetilde{A} := \varinjlim\widetilde{A}_n$ be the inductive limit $C^*$-algebra. Moreover, since 
\[
\xymatrix{ 
A_n \ar@{^{(}->}[r] \ar[d]^{\pi_{\alpha_n}} & A_{n+1} \ar[d]^{\pi_{\alpha_{n+1}}} \\
\pi_{\alpha_n}(A_n) \ar@{^{(}->}[r] \ar@{^{(}->}[d] & \pi_{\alpha_{n+1}}(A_{n+1}) \ar@{^{(}->}[d] \ar@{}[lu]|{\circlearrowleft} \ar@{}[ld]|{\circlearrowleft}\\
\widetilde{A}_n \ar@{^{(}->}[r] & \widetilde{A}_{n+1}, 
}
\]
there is a unique injective $*$-homomorphism $\pi_\alpha := \varinjlim \pi_{\alpha_n}  : A = \varinjlim A_n \to \widetilde{A} = \varinjlim \widetilde{A}_n$ such that $\pi_\alpha(a) = \pi_{\alpha_m}(a)$ in $\widetilde{A}$ for every $a \in A_n$ and $m\geq n$. By \eqref{Eq3.3} and \eqref{Eq3.4} we can take the inductive limit action $\widetilde{\alpha} := \varinjlim\widetilde{\alpha}_n : \Gamma \curvearrowright \widetilde{A}$, which acts on $\pi_\alpha(A)$ trivially. 

\begin{definition}\label{D3.1} We then call $(\widetilde{\alpha} : \Gamma \curvearrowright \widetilde{A} = \varinjlim \widetilde{A}_n)$ as above the \emph{$\rho$-extension} of $(A,\alpha^t) = \varinjlim (A_n,\alpha_n^t)$. 
\end{definition}   

We remark that $\Gamma$ is not a canonical object of the flow $\alpha^t$ because it depends on the choice of $\rho_z$'s. In the next section, we will select $\Gamma$ to be a canonical object under an additional assumption on $A = \varinjlim A_n$.  

\medskip
Following a standard strategy in operator algebras dating back to Takesaki's structure theorem for type III factors (see e.g., \cite[section XII.1]{Takesaki:Book2}), we will interpret $K_\beta^\mathrm{ln}(\alpha^t)$ as a suitable class of tracial weights on $\widetilde{A}$.   

\medskip
We start with necessary concepts/facts on (tracial) weights on $C^*$-algebras (see \cite[chapter VII]{Takesaki:Book2} as well as \cite[section V.2]{Takesaki:Book1}). A \emph{weight} $\psi$ on a $C^*$-algebra $C$ means a map from $C_+$ to $[0,+\infty]$ such that 
\begin{align*} 
\psi(c_1+c_2) &= \psi(c_1)+\psi(c_2), \qquad c_1,c_2 \in C_+ \\
\psi(t c) &= t\psi(c), \qquad t \in [0,+\infty), \quad c \in C_+
\end{align*} 
with convention $0\times(+\infty)=0$. We call $\psi$ a \emph{tracial weight} if, in addition, $\psi(c^* c) = \psi(cc^*)$ holds for any $c \in C$. The \emph{definition domain} $\mathfrak{m}_\psi$ of $\psi$ is defined to be the linear span of all the $c_1^* c_2$ with $\psi(c_k^* c_k) < +\infty$, $k=1,2$. By the polarization identity we can extend $\psi$ to $\mathfrak{m}_\psi$ as a linear functional. When $\psi$ is tracial, $\psi$ enjoys that $\psi(c_1 c_2) = \psi(c_2 c_1)$ if one of $c_i \in C$ falls into $\mathfrak{m}_\psi$; see the proof of \cite[Lemma V.2.16]{Takesaki:Book1}. When $C$ is a $W^*$-algebra, $\psi$ is said to be \emph{normal} if $c_i \nearrow c$ in $C_+$ implies $\psi(c_i) \nearrow \psi(c)$, and also \emph{semifinite} if $C$ is generated as a $W^*$-algebra by all the $c \in C_+$ with $\psi(c) < +\infty$.     

\begin{definition}\label{D3.2} 
{\rm(1)} An \emph{$(\widetilde{\alpha}^\gamma,\beta)$-scaling trace} is defined to be a tracial weight $\tau : (\widetilde{A})_+ \to [0,\infty]$ such that 
\begin{itemize}
\item[(i)] For each $x \in \widetilde{A}$ and each $n$, the mapping $y \in (\widetilde{A}_n)_+ \mapsto \tau(xyx^*) \in [0,+\infty]$ is normal,  
\item[(ii)] $\tau\circ\widetilde{\alpha}^\gamma = \gamma^\beta\,\tau$ for all $\gamma \in \Gamma$,  
\item[(iii)] $\tau(e_1) = 1$. 
\end{itemize}
All the $(\widetilde{\alpha}^\gamma,\beta)$-scaling traces are denoted by $TW_\beta^\mathrm{ln}(\widetilde{\alpha}^\gamma)$. 

{\rm(2)} We define a normal semifinite weight $\mathrm{tr}_\beta : (\widetilde{A}_0)_+ \to [0,\infty]$ by $\mathrm{tr}_\beta(e_\gamma) = \gamma^\beta$ for every $\gamma \in \Gamma$. 
\end{definition} 

Remark that items (iii),(iv) in (1) imply that $\tau(e_\gamma) = \gamma^\beta$ for every $\gamma \in \Gamma$ so that $\tau$ is  semifinite on each $\widetilde{A}_n$. 
In fact, letting $e_\mathcal{F} := \sum_{\gamma\in\mathcal{F}} e_\gamma$ with $\mathcal{F} \Subset \Gamma$ we see that $\bigcup_{\gamma \in \mathcal{F}} e_\mathcal{F} (\widetilde{A}_n)_+ e_\mathcal{F}$ is $\sigma$-weakly dense in $(\widetilde{A}_n)_+$ and items (iii),(iv) imply $0\leq \tau(e_\mathcal{F} x e_\mathcal{F}) \leq \Vert x\Vert\,\sum_{\gamma\in\mathcal{F}}\gamma^\beta < +\infty$ for any $x \in (\widetilde{A}_n)_+$.

\begin{lemma}\label{L3.2} For each $\omega \in K_\beta^{\mathrm{ln}}(\alpha^t)$, the restriction of $\omega\,\bar{\otimes}\,\mathrm{id}  : A_n\,\bar{\otimes}\,B(L^2(G)) \to \mathbb{C}1\,\bar{\otimes}\,B(L^2(G))$ (the composition of $x \mapsto 1\otimes x$ and the normal slice map $R_\omega : A\,\bar{\otimes}\,B(L^2(G))$ sending $a\otimes x$ to $\omega(a)x$, see e.g., \cite{Tomiyama:PJM69}) to $\widetilde{A}_n = A_n\,\bar{\rtimes}_{\alpha_n}\,G$ defines a unique normal conditional expectation $\widetilde{E}_{\omega,n} : \widetilde{A}_n \twoheadrightarrow \widetilde{A}_0$ such that $\widetilde{E}_{\omega,n}(\pi_{\alpha_n}(a)) = \omega(a)1$ for every $a \in A_n$. Then, $\widetilde{E}_{\omega,n+1}$ coincides with $\widetilde{E}_{\omega,n}$ on $\widetilde{A}_n$, and the inductive limit conditional expectation $\widetilde{E}_\omega := \varinjlim \widetilde{E}_{\omega,n}$ from $\widetilde{A}=\varinjlim \widetilde{A}_n$ onto $\widetilde{A}_0$ is well defined. 
\end{lemma}
\begin{proof} 
Since the image of $\mathbb{R}$ in $G$ is dense and $\omega\circ\alpha^t = \omega$ for all $t \in \mathbb{R}$, we have $\omega\circ\alpha_n^g(a) = \omega(a)$ for all $g \in G$ and $a \in A_n$. By \cite[Theorem 2.5.31(a)]{BratteliRobinson:Book1} we can choose a representing vector $\xi \in \mathcal{K}_n$ of the restriction of $\omega$ to $A_n$, so that $\omega(a) = (a\,\xi\,|\,\xi)_{\mathcal{K}_n}$ holds for every $a \in A_n$. By definition we observe that $(R_\omega(x)f_1\,|\,f_2)_{L^2(G)} = (x\,\xi\otimes f_1\,|\,\xi\otimes f_2)_{\mathcal{K}_n\,f_2\bar{\otimes}\,L^2(G)}$ for all $x \in A_n\,\bar{\otimes}\,B(L^2(G))$ and $f_1,f_2 \in L^2(G)$. By the identification $L^2(G;\mathcal{K}_n) = \mathcal{K}_n\,\bar{\otimes}\,L^2(G)$,  we have 
\begin{align*}
(\pi_{\alpha_n}(a)\,\xi\otimes f_1\,|\,\xi\otimes f_2)_{\mathcal{K}_n\,\bar{\otimes}\,L^2(G)} 
&= 
\int_G (\alpha_n^{g^{-1}}(a)\xi\,|\,\xi)_{\mathcal{K}_n}\,f_1(g)\overline{f_2(g)}\,dg \\
&= 
\int_G \omega(\alpha_n^{g^{-1}}(a))\,f_1(g)\overline{f_2(g)}\,dg \\
&= 
\omega(a)\,(f_1\,|\,f_2)_{L^2(G)}
\end{align*}
for all $a \in A_n$ and $f_1,f_2 \in C(G) \subset L^2(G)$. We conclude that $R_\omega(\pi_{\alpha_n}(a)) = \omega(a)\,1_{L^2(G)}$ and hence $(\omega\,\bar{\otimes}\,\mathrm{id})(\pi_{\alpha_n}(a)) = \omega(a)\,1$ for all $a \in A_n$. Since the $\pi_{\alpha_n}(a)\lambda(g)$ form a $\sigma$-weakly total subset of $\widetilde{A}_n$, it follows that $(\omega\,\bar{\otimes}\,\mathrm{id})(\widetilde{A}_n) = \widetilde{A}_0$ and hence the restriction of $\omega\,\bar{\otimes}\,\mathrm{id}$ to $\widetilde{A}_n$ gives the desired conditional expectation $\widetilde{E}_{\omega,n}$. The rest of assertion is now obvious. 
\end{proof} 

\begin{lemma}\label{L3.4} For each $\omega \in K_\beta^\mathrm{ln}(\alpha^t)$ the weight $\tau_\omega := \mathrm{tr}_\beta\circ \widetilde{E}_\omega : \widetilde{A}_+ \to [0,\infty]$ becomes an $(\widetilde{\alpha}^\gamma,\beta)$-scaling trace. 
\end{lemma}
\begin{proof} 
We have to confirm that $\tau_\omega$ satisifies items (i)-(iv) of Definition \ref{D3.2}(1). 

\medskip
We remark that the restriction of $\omega$ to $A_n$ becomes $\sum_{z \in \mathfrak{Z}_n} \omega(z)\,\tau_z^\beta$ (see \cite[Lemma 7.3]{Ueda:preprint20}). We set $s := \sum_{z \in \mathfrak{Z}_n} \mathbf{1}_{(0,1]}(\omega(z))\,z \in \mathcal{Z}(A_n)$, which is the support projection of the restriction of $\omega$ to $A_n$, that is, $\omega$ is faithful on $sA_n$ and identically zero on $(1-s)A_n$. One can easily confirm that $\omega$ enjoys the $(\alpha_n^{-\beta t},-1)$-KMS condition, and hence the restriction of $\alpha_n^{-\beta t}$ to $sA_n$ gives the modular automorphism group associated with the restriction of $\omega$ to $sA_n$ by \cite[Theorem 5.3.10]{BratteliRobinson:Book2}. 

We observe that $\pi_{\alpha_n}(s) = s\otimes1 \in \mathcal{Z}(\widetilde{A}_n)$ and hence $\pi_{\alpha_n}(s)\widetilde{A}_n = (sA_n)\bar{\rtimes}_{\alpha_n}G \subset (sA_n)\,\bar{\otimes}\,B(L^2(G))$ by its construction. We have a bijective $*$-homomorphism $\iota : \pi_{\alpha_n}(s)\widetilde{A}_0 \to \widetilde{A}_0$ sending $\lambda^{(0)}(g) := \pi_{\alpha_n}(s)\lambda(g) = s\otimes\lambda_g$ to $1\otimes\lambda_g = \lambda(g)$ for any $g \in G$. With 
\[
e_\gamma^{(00)} := \int_G \overline{\langle \gamma,g\rangle}\lambda_g\,dg, \qquad \gamma \in \Gamma,
\] 
we observe that the bijective $*$-homomorphism $\iota$ sends $e_\gamma^{(0)} := s\otimes e^{(00)}_\gamma$ to $1\otimes e^{(00)}_\gamma = e_\gamma$ for every $\gamma \in \Gamma$. For a while, we will work with $\pi_{\alpha_n}(s)\widetilde{A}_n = (sA_n)\bar{\rtimes}_{\alpha_n^g}G$ whose generators are $\pi_{\alpha_n}(a)$ ($a \in sA_n$) as well as $\lambda^{(0)}(g)$ ($g \in G$) or $e^{(0)}_\gamma$ ($\gamma \in \Gamma$) along the line of proof of \cite[Theorem 1]{Ueda:AJM16}. 

Let $\tilde{\omega}$ be the dual weight on $(sA_n)\bar{\rtimes}_{\alpha_n^g}G$ constructed out of the restriction of $\omega$ to $sA_n$ (see \cite[Definition X.1.16, Lemma X.1.18]{Takesaki:Book2}), which satisfies that
\[
\tilde{\omega}\Big(\Big(\int_G \lambda^{(0)}(g)\pi_{\alpha_n}(a(g))\,dg\Big)^*\Big(\int_G \lambda^{(0)}(g)\pi_{\alpha_n}(a(g))\,dg\Big)\Big) 
= 
\int_G \omega(a(g)^* b(g))\,dg
\]
for any $\sigma$-strong$^*$-continuous functions $a,b : G \to sA_n$, where $\tilde{\omega}$ extends to its definition domain $\mathfrak{m}_{\tilde{\omega}}$. Moreover, its modular automorphism $\sigma_t^{\tilde{\omega}}$ enjoys that 
\[
\sigma_t^{\tilde{\omega}}(\pi_{\alpha_n}(a)) = \pi_{\alpha_n}(\alpha_n^{-\beta t}(a)), \qquad \sigma_t^{\tilde{\omega}}(\lambda^{(0)}(g)) = \lambda^{(0)}(g)
\]
for all $a \in sA_n$ and $g \in G$. In particular, we obtain $\sigma_t^{\widetilde{\omega}} = \mathrm{Ad}\lambda^{(0)}(-\beta t)$ for every $t \in \mathbb{R}$. Also, we have $\tilde{\omega}(e^{(0)}_\gamma) = \tilde{\omega}(e^{(0)}_\gamma e^{(0)}_\gamma) = \int_G dg = 1$, and hence the restriction of $\tilde{\omega}$ to $\lambda^{(0)}(G)''$ is semifinite. Thus, Takesaki's theorem \cite[Theorem IX.4.2]{Takesaki:Book2} guarantees that there is a unique faithful normal conditional expectation $E : (sA_n) \bar{\rtimes}_{\alpha_n^g}G \to \lambda^{(0)}(G)''$ with $\tilde{\omega}\circ E = \tilde{\omega}$. Then  
\begin{align*}
\tilde{\omega}(E(\pi_{\alpha_n}(a)) e^{(0)}_\gamma) 
=
\tilde{\omega}\circ E(e^{(0)}_\gamma \pi_{\alpha_n}(a) e^{(0)}_\gamma) 
=
\tilde{\omega}(e^{(0)}_\gamma\pi_{\alpha_n}(a) e^{(0)}_\gamma) 
=
\int_G \omega(a)\,dg 
= 
\omega(a)\,\tilde{\omega}(e^{(0)}_\gamma),
\end{align*}
implying that $E(\pi_{\alpha_n}(a)) = \omega(a)1$ for every $a \in sA_n$ because $\tilde{\omega}(e^{(0)}_\gamma) = 1$. Since 
\[
\lambda^{(0)}(-\beta t) = \sum_{\gamma \in \Gamma} \langle\gamma,-\beta t\rangle e^{(0)}_\gamma = \sum_{\gamma \in \Gamma} \gamma^{i(-\beta t)}\,e^{(0)}_\gamma 
= \Big(\sum_{\gamma\in\Gamma}\gamma^{-\beta}\,e^{(0)}_\gamma\Big)^{it} =: H^{it} 
\]
($H$ is a non-singular positive self-adjoint operator affiliated with $\lambda^{(0)}(G)''$), \cite[Theorem VIII.3.14]{Takesaki:Book2} and its proof show that a semifinite normal tracial weight on $(sA_n)\bar{\rtimes}_{\alpha_n^g}G$ can be defined to be $\tilde{\omega}(H^{-1}(\,\cdot\,))$ (which needs some justification; see \cite[Lemma VIII.2.8]{Takesaki:Book2}). Then we can easily verify $\tilde{\omega}(H^{-1}E(\,\cdot\,)) = \tilde{\omega}(H^{-1}(\,\cdot\,))$, since $H$ is affiliated with $\lambda^{(0)}(G)''$. We observe that $H^{-1}e^{(0)}_\gamma = \gamma^\beta\,e^{(0)}_\gamma$ and hence $\tilde{\omega}(H^{-1}e^{(0)}_\gamma) = \gamma^\beta\,\tilde{\omega}(e^{(0)}_\gamma) = 
\gamma^\beta$ for every $\gamma \in \Gamma$. 

Since
\[
\widetilde{E}_{\omega,n}(\pi_{\alpha_n}(a)\lambda(g)) = \omega(a)\lambda(g) = \omega(sa)\iota(\lambda^{(0)}(g)) = \iota(E(\pi_{\alpha_n}(s)\pi_{\alpha_n}(a)\lambda^{(0)}(g)))
\] 
for any $a \in A_n$ and $g \in G$, we have $\widetilde{E}_{\omega,n}(x) = \iota(E(\pi_{\alpha_n}(s)x))$ for every $x \in \widetilde{A}_n$. Since $\mathrm{tr}_\beta(\iota(e^{(0)}_\gamma)) = \mathrm{tr}_\beta(e_\gamma) = \gamma^\beta = \tilde{\omega}(H^{-1} e^{(0)}_\gamma)$ for every $\gamma \in \Gamma$, we also have $\mathrm{tr}_\beta\circ\iota = \tilde{\omega}(H^{-1}(\,\cdot\,))$ on $(\widetilde{A}_0)_+$. Therefore, we obtain 
\[
\mathrm{tr}_\beta\circ\widetilde{E}_{\omega,n}(x) = \mathrm{tr}_\beta(\iota(E(\pi_{\alpha_n}(s)x))) = \tilde{\omega}(H^{-1}E(\pi_{\alpha_n}(s)x)) = \tilde{\omega}(H^{-1}\pi_{\alpha_n}(s)x)
\] 
for every $x \in (\widetilde{A}_n)_+$. Since $\tau_\omega$ coincides with $\mathrm{tr}_\beta\circ\widetilde{E}_{\omega,n}$ on $\widetilde{A}_n$, it must be a normal semifinite tracial weight on $\widetilde{A}_n$. 

\medskip
Let $x \in \widetilde{A}$ be arbitrarily chosen. Choose a sequence $x_k \in \bigcup_{n\geq0}\widetilde{A}_n$ in such a way that $\Vert x_k - x\Vert \to 0$ as $k\to\infty$. 

For any net $y_\lambda \nearrow y$ in $(\widetilde{A}_n)_+$ we have 
\begin{align*} 
\limsup_\lambda |\phi(E_\omega(x y_\lambda x^*) - E_\omega(x y x^*))| 
\leq 2\Vert\phi\Vert\,\Vert y\Vert\,(\Vert x\Vert + \Vert x_k\Vert)\,\Vert x_k - x\Vert \overset{k\to\infty}{\to} 0
\end{align*}
for every normal linear functional $\phi$ on $\widetilde{A}_0$, since the $x_k y_\lambda x_k^*$ and $x_k y x_k^*$ fall into some $\widetilde{A}_m$ with $m \geq n$ for a fixed $k$ and since the restriction of $E_\omega$ to $\widetilde{A}_m$ is normal. Hence we conclude that $E_\omega(x y_\lambda x^*) \nearrow E_\omega(xyx^*)$, that is,  $y \in \widetilde{A}_0 \mapsto E_\omega(xyx^*) \in \widetilde{A}_0$ is a normal map. It follows that $\tau_\omega = \mathrm{tr}_\beta\circ E_\omega$ enjoys item (i) thanks to the normality of $\mathrm{tr}_\beta$. 

Let $\mathcal{F}_1,\mathcal{F}_2 \Subset \Gamma$ be arbitrarily given. For each $k$, $e_{\mathcal{F}_2}x_k e_{\mathcal{F}_1}$ falls in some $\widetilde{A}_n$, and what we have proved above shows that $\tau_\omega(e_{\mathcal{F}_1}x_k^* e_{\mathcal{F}_2} x_k e_{\mathcal{F}_1}) 
= 
\tau_\omega(e_{\mathcal{F}_2} x_k e_{\mathcal{F}_1} x_k^* e_{\mathcal{F}_2})$,
since $\tau_\omega$ coincides with $\mathrm{tr}_\beta\circ \widetilde{E}_{\omega,n}$ on $\widetilde{A}_n$. By the dominated convergence theorem (n.b., $\widetilde{A}_0 \cong \ell^\infty(\Gamma)$ is pointed out before), we have  
\begin{align*}
\tau_\omega(e_{\mathcal{F}_1}x_k^* e_{\mathcal{F}_2} x_k e_{\mathcal{F}_1}) 
= 
\mathrm{tr}_\beta(\widetilde{E}_\omega(x_k^* e_{\mathcal{F}_2} x_k) e_{\mathcal{F}_1}) 
&\to 
\mathrm{tr}_\beta(\widetilde{E}_\omega(x^* e_{\mathcal{F}_2} x) e_{\mathcal{F}_1}) = 
\tau_\omega(e_{\mathcal{F}_1}x^* e_{\mathcal{F}_2} x e_{\mathcal{F}_1}),  \\
\tau_\omega(e_{\mathcal{F}_2}x_k e_{\mathcal{F}_1} x_k^* e_{\mathcal{F}_2}) 
=
\mathrm{tr}_\beta(\widetilde{E}_\omega(x_k e_{\mathcal{F}_1} x_k^*) e_{\mathcal{F}_2}) 
&\to 
\mathrm{tr}_\beta(\widetilde{E}_\omega(x e_{\mathcal{F}_1} x^*) e_{\mathcal{F}_2}) 
= 
\tau_\omega(e_{\mathcal{F}_2}x e_{\mathcal{F}_1} x^* e_{\mathcal{F}_2})
\end{align*} 
as $k\to\infty$. Consequently, we obtain that $\tau_\omega(e_{\mathcal{F}_1}x^* e_{\mathcal{F}_2} x e_{\mathcal{F}_1}) = \tau_\omega(e_{\mathcal{F}_2}x e_{\mathcal{F}_1} x^* e_{\mathcal{F}_2})$ for any $\mathcal{F}_1,\mathcal{F}_2 \Subset \Gamma$. 

By the normality of $\mathrm{tr}_\beta$, we have 
\[
\tau_\omega(e_{\mathcal{F}_1}x^* e_{\mathcal{F}_2} x e_{\mathcal{F}_1}) = \mathrm{tr}_\beta(\widetilde{E}_\omega(x^* e_{\mathcal{F}_2} x) e_{\mathcal{F}_1}) \nearrow \mathrm{tr}_\beta(\widetilde{E}_\omega(x^* e_{\mathcal{F}_2} x)) = \tau_\omega(x^* e_{\mathcal{F}_2} x)
\]
as $\mathcal{F}_1 \nearrow \Gamma$. On the other hand, we have, by item (i), $\tau_\omega(e_{\mathcal{F}_2}x e_{\mathcal{F}_1} x^* e_{\mathcal{F}_2}) \nearrow \tau_\omega(e_{\mathcal{F}_2}xx^* e_{\mathcal{F}_2})$ as $\mathcal{F}_1 \nearrow \Gamma$. Hence $\tau_\omega(x^* e_{\mathcal{F}_2} x) = \tau_\omega(e_{\mathcal{F}_2}xx^* e_{\mathcal{F}_2})$ for any $\mathcal{F}_2 \Subset \Gamma$. Similarly,  taking the limit as $\mathcal{F}_2 \nearrow \Gamma$ we obtain $\tau_\omega(x^* x) = \tau_\omega(xx^*)$. Hence $\tau_\omega$ is a tracial weight. 

\medskip
We have 
\begin{align*}
\widetilde{E}_\omega\circ\widetilde{\alpha}^\gamma(\pi_\alpha(a)\lambda(g)) 
&= 
\overline{\langle \gamma,g\rangle}\,\widetilde{E}_\omega(\pi_\alpha(a)\lambda(g)) 
= 
\overline{\langle \gamma,g\rangle}\,\widetilde{E}_{\omega,n}(\pi_{\alpha_n}(a)\lambda(g)) \\
&= 
\overline{\langle \gamma,g\rangle}\,\omega(a)\,\lambda(g) 
= 
\widetilde{\alpha}^\gamma(\widetilde{E}_{\omega,n}(\pi_{\alpha_n}(a)\lambda(g))) 
= 
\widetilde{\alpha}^\gamma\circ\widetilde{E}_\omega(\pi_\alpha(a)\lambda(g))
\end{align*}
for any $a \in A_n$ and $g \in G$. Hence we obtain $\widetilde{E}_\omega\circ\widetilde{\alpha}^\gamma = \widetilde{\alpha}^\gamma\circ\widetilde{E}_\omega$ for every $\gamma \in \Gamma$. Moreover, we observe that 
$
\mathrm{tr}_\beta\circ\widetilde{\alpha}^\gamma(e_{\gamma'}) = \mathrm{tr}_\beta(e_{\gamma\gamma'}) = \gamma^\beta \gamma'{}^\beta = \gamma^\beta\,\mathrm{tr}_\beta(e_{\gamma'})
$
for all $\gamma, \gamma' \in \Gamma$. Therefore, we obtain that $\mathrm{tr}_\beta\circ\widetilde{\alpha}^\gamma = \gamma^\beta\,\mathrm{tr}_\beta$, and thus $\tau_\omega$ satisfies item (ii). Item (iii) is trivial by Definition \ref{D3.2}(2). 
\end{proof}

\begin{lemma}\label{L3.5} For each $\tau \in TW_\beta^\mathrm{ln}(\widetilde{\alpha}^\gamma)$ the mapping 
\[
a \in A_+ \mapsto \tau(e_1\pi_\alpha(a)) = \tau(\pi_\alpha(a)e_1) = \tau(e_1\pi_\alpha(a)e_1) \in [0,\infty)
\]
extends to the whole $A$ and defines an element of $K_\beta^\mathrm{ln}(\alpha^t)$. 
\end{lemma}
\begin{proof}
Since $\tau(e_1) < +\infty$, $\tau(e_1 \pi_{\alpha_n}(a)) = \tau(\pi_{\alpha_n}(a)e_1) = \tau(e_1 \pi_{\alpha_n}(a) e_1)$ makes sense for all $a \in A$. By the standard Phragmen--Lindel\"{o}f method, it suffices to show that $\tau(e_1 \pi_{\alpha_n}(ab)) = \tau(\pi_{\alpha_n}(b\alpha_n^{i\beta}(a))e_1)$ ($=\tau(e_1\pi_{\alpha_n}(b\alpha_n^{i\beta}(a)))$) for any $\alpha_n^t$-analytic $a \in A_n$ and any $b \in A_n$. 

For each $\gamma \in \Gamma$ we define $E_\gamma^{(n)} : A_n \to A_n$ by 
\[
E_\gamma^{(n)}(a) := \int_G \overline{\langle \gamma,g\rangle}\, \alpha_n^g(a)\,dg, \qquad a \in A_n. 
\]
Then we have 
\begin{equation}\label{Eq3.5}
E_\gamma^{(n)}(a)^* = E_{\gamma^{-1}}^{(n)}(a^*)
\end{equation}
for every $a \in A_n$. Remark that $E_\gamma^{(n)}(\alpha_n^t(a)) = \gamma^{it}\,E_\gamma^{(n)}(a)$ for every $a \in A_n$, and moreover that $z  \mapsto E_\gamma^{(n)}(\alpha_n^z(a))$ is entire for every $\alpha_n^t$-analytic $a \in A_n$ ({\it n.b.,} this can easily be confirmed by using \cite[appendix A1]{Takesaki:Book2}). By the unicity theorem in complex analysis we conclude that 
\begin{equation}\label{Eq3.6}
\gamma^{-\beta}E_\gamma^{(n)}(a) = E_\gamma^{(n)}(\alpha_n^{i\beta}(a))
\end{equation} 
for every $\alpha_n^t$-analytic $a \in A_n$. We also observe that 
\begin{equation}\label{Eq3.7}
e_1\pi_{\alpha_n}(a)e_\gamma = \pi_{\alpha_n}(E_{\gamma^{-1}}^{(n)}(a))e_\gamma
\end{equation}
for every $a \in A_n$. Taking the adjoint of this identity together with \eqref{Eq3.5} we also have
\begin{equation}\label{Eq3.8}
e_\gamma \pi_{\alpha_n}(a) e_1 = e_\gamma \pi_{\alpha_n}(E_\gamma^{(n)}(a))
\end{equation}
for every $a \in A_n$. 

Let $a \in A_n$ be an arbitrary $\alpha_n^t$-analytic element, and $b \in A_n$ be an arbitrary element of $A_n$. Then we have
\begin{align*} 
\tau(e_1 \pi_{\alpha_n}(ab)) 
&= 
\tau(e_1\pi_{\alpha_n}(a)\pi_{\alpha_n}(b)e_1) 
=
\tau(\pi_{\alpha_n}(b) e_1 \pi_{\alpha_n}(a)) \qquad \text{(trace property)}\\
&= 
\sum_{\gamma \in \Gamma} \tau(\pi_{\alpha_n}(b) e_1 \pi_{\alpha_n}(a) e_\gamma) \\
&= 
\sum_{\gamma \in \Gamma} \tau(\pi_{\alpha_n}(b E_{\gamma^{-1}}^{(n)}(a)) e_\gamma) \qquad\quad \text{(use \eqref{Eq3.7})}\\
&= 
\sum_{\gamma \in \Gamma} \tau\circ\widetilde{\alpha}^\gamma(\pi_{\alpha_n}(b E_{\gamma^{-1}}^{(n)}(a)) e_1) \quad \text{(use \eqref{Eq3.3})}\\
&= 
\sum_{\gamma \in \Gamma} \gamma^\beta\,\tau(\pi_{\alpha_n}(b E_{\gamma^{-1}}^{(n)}(a)) e_1) \quad \text{(use item (ii) in Defition \ref{D3.2}(1))}\\
&= 
\sum_{\gamma \in \Gamma} \tau(\pi_{\alpha_n}(b (\gamma^\beta\,E_{\gamma^{-1}}^{(n)}(a))) e_1) \\
&= 
\sum_{\gamma \in \Gamma} \tau(\pi_{\alpha_n}(b E_{\gamma^{-1}}^{(n)}(\alpha_n^{i\beta}(a)) e_1) \qquad\quad \text{(use \eqref{Eq3.6})}\\
&=
\sum_{\gamma \in \Gamma} \tau(\pi_{\alpha_n}(b) e_{\gamma^{-1}}\pi_{\alpha_n}(\alpha_n^{i\beta}(a)) e_1) \quad\ \text{(use \eqref{Eq3.8})}\\
&=
\tau(\pi_{\alpha_n}(b\alpha_n^{i\beta}(a)) e_1)
\end{align*}
Hence we are done.
\end{proof} 

So far, we have constructed two maps 
\begin{equation}\label{Eq3.9}
\begin{aligned}
\omega \in K_\beta^\mathrm{ln}(\alpha^t) 
&\mapsto 
\tau_\omega = \mathrm{tr}_\beta\circ \widetilde{E}_\omega \in TW_\beta^\mathrm{ln}(\widetilde{\alpha}^\gamma), \\
\tau \in TW_\beta^\mathrm{ln}(\widetilde{\alpha}^\gamma) 
&\mapsto 
(a \mapsto \omega_\tau(a):=\tau(e_1 \pi_\alpha(a))) \in K_\beta^\mathrm{ln}(\alpha^t). 
\end{aligned}
\end{equation}
Since $\tau_\omega(e_1 \pi_\alpha(a)) = \omega(a)$ for all $a \in A$, it follows that the first map in \eqref{Eq3.9} is injective. We also remark that $\omega_\tau$ in \eqref{Eq3.9} makes sense on the whole $A$ since $\tau(e_1)<+\infty$.

\begin{lemma}\label{L3.6} 
We have $\tau = \tau_{\omega_\tau}$ for every $\tau \in TW_\beta^\mathrm{ln}(\widetilde{\alpha}^\gamma)$. 
\end{lemma} 
\begin{proof}
For any $a\in A_+$, $g \in G$ and $\gamma \in \Gamma$ we have 
\begin{align*}
\tau(\pi_\alpha(a)\lambda(g) e_\gamma) 
&=
\tau(\pi_\alpha(a)\,\langle \gamma,g\rangle\,e_\gamma) 
=
\langle \gamma,g\rangle\,\tau(e_\gamma \pi_\alpha(a)e_\gamma) \\
&=
\langle \gamma,g\rangle\,\tau\circ\widetilde{\alpha}^\gamma(e_1 \pi_\alpha(a)e_1) 
=
\langle \gamma,g\rangle\,\gamma^\beta\,\tau(e_1 \pi_\alpha(a)e_1) \\ 
&=
\langle \gamma,g\rangle\,\gamma^\beta\,\omega_\tau(a)  
=
\langle \gamma,g\rangle\,\omega_\tau(a)\,\mathrm{tr}_\beta(e_\gamma) 
= 
\mathrm{tr}_\beta(\widetilde{E}_{\omega_\tau}(\pi_\alpha(a))\lambda(g)e_\gamma).
\end{align*}
It follows that $\tau(x e_\gamma) = \tau_{\omega_\tau}(x e_\gamma)$ holds for any $x \in \widetilde{A}$ and $\gamma \in \Gamma$. Therefore, we have 
$\tau(x e_\mathcal{F}) = \tau_{\omega_\tau}(x e_\mathcal{F})$ for any $x \in \widetilde{A}$ and any finite $\mathcal{F} \Subset \Gamma$. By the trace property together with $\tau(e_\mathcal{F}) < +\infty$ we have, by item (i) of Definition \ref{D3.2}(1),  
\[
\tau(x e_\mathcal{F}) = \tau(e_\mathcal{F} x e_\mathcal{F}) = \tau(x^{1/2}e_\mathcal{F} x^{1/2}) \nearrow \tau(x)
\]
as $\mathcal{F} \nearrow \Gamma$ for every $x \in \widetilde{A}_+$. We also have $\tau_{\omega_\tau}(xe_\mathcal{F}) = \mathrm{tr}_\beta(\widetilde{E}_{\omega_\tau}(x)e_\mathcal{F}) \nearrow \mathrm{tr}_\beta(\widetilde{E}_{\omega_\tau}(x)) = \tau_{\omega_\tau}(x)$ as $\mathcal{F} \nearrow \Gamma$ for every $x \in \widetilde{A}_+$. We conclude that $\tau = \tau_{\omega_\tau}$ holds. 
\end{proof}

Summing up the discussions so far we have obtained the following theorem: 

\begin{theorem}\label{T3.7} 
The maps in \eqref{Eq3.9} are inverse to each other. Therefore, $K_\beta^\mathrm{ln}(\alpha^t)$ and $TW_\beta^\mathrm{ln}(\widetilde{\alpha}^\gamma)$ are affine-isomorphic. 
\end{theorem} 

Thanks to the theorem, a natural topology on $TW_\beta^\mathrm{ln}(\widetilde{\alpha}^\gamma)$ is defined by the following convergence: $\tau_i \to \tau$ in $TW_\beta^\mathrm{ln}(\widetilde{\alpha}^\gamma)$ means that $\tau_i(e_1\pi_\alpha(a))\to\tau(e_1\pi_\alpha(a))$ for every $a \in A$. By item (ii) of Definition \ref{D3.2}(1) we have $\tau_i \to \tau$ in $TW_\beta^\mathrm{ln}(\widetilde{\alpha}^\gamma)$ implies that $\tau_i(e_\mathcal{F} x) \to \tau(e_\mathcal{F}x)$ for any $\mathcal{F} \Subset \Gamma$ and $x \in \widetilde{A}$, and hence $\liminf_i \tau_i(x) \geq \tau(x)$ for all $x \in \widetilde{A}_+$. 

\section{Weight-extended branching graph}\label{S4} 

In the previous section, we transferred the study of locally normal $(\alpha^t,\beta)$-KMS states to that of $(\widetilde{ \alpha}^\gamma,\beta)$-scaling traces on $\widetilde{A}=\varinjlim \widetilde{A}_n$. Here, we will translate this procedure into terms of standard link. For this purpose, we have to assume that all $\dim(z) < \infty$. Then we can select each $\rho_z$ in such a way that $\mathrm{Tr}(\rho_z) = \mathrm{Tr}(\rho_z^{-1})$. Under this selection, the $\rho = \{\rho_z\}_{z\in\mathfrak{Z}}$ is uniquely determined from the flow $\alpha^t$, and hence both $\Gamma=\Gamma(\rho)$ and $G=\widehat{\Gamma}$ are canonical objects associated with $\alpha^t$. Hence we call this $\Gamma$ the \emph{weight group}, and the $\rho$-extension $(\widetilde{\alpha} : \Gamma \curvearrowright \widetilde{A}=\varinjlim \widetilde{A}_n)$ the \emph{weight-extension} in this case. Remark that this choice of $\Gamma$ is not exactly the same as that in the so-called discrete decomposition for type III factors due to Connes (see e.g., \cite{Ueda:AJM16} whose treatment on it fits the present discussion). 

\subsection{Weight-extended branching graph}\label{S4.1} 
Let 
\[
\rho_z = \sum_{\gamma\in\Gamma} \gamma\,p_z(\gamma)
\]
be the spectral decomposition ({\it n.b.,} the support of $p_z(\,\cdot\,)$ is a finite subset of $\Gamma$ due to $\dim(z)<+\infty$). Then 
\[
u_z(g) = \sum_{\gamma\in\Gamma} \langle \gamma,g\rangle\,p_z(\gamma), \qquad g\in G, 
\]
and regarding $p_z(\gamma)$, $u_z(g)$ as elements of $zA_n \subset A_n$ we have 
\[
u_n(g) = \sum_{z\in\mathfrak{Z}_n} u_z(g) = \sum_{z\in\mathfrak{Z}_n}\sum_{\gamma \in \Gamma} \langle\gamma,g\rangle\,p_z(\gamma) \in A_n, \qquad g \in G.  
\]
The unitary operator $U$ on $L^2(G;\mathcal{K}_n)$ defined by 
\[
(U\xi)(g) = u_n(g)\xi(g), \qquad \xi\in C(G;\mathcal{K}_n) \subset L^2(G;\mathcal{K}_n)
\] 
enjoys that 
\begin{equation}\label{Eq4.1}
U\pi_{\alpha_n}(a)U^* = a\otimes1, \qquad U\lambda(g)U^* = u_n(g)\otimes\lambda_g 
\end{equation}
for any $a \in A_n$ and $g \in G$, where we identify $L^2(G;\mathcal{K}_n) = \mathcal{K}_n\,\bar{\otimes}\,L^2(G)$ as in section \ref{S3}. See e.g., \cite[Theorem X.1.7(ii)]{Takesaki:Book2}.  We observe that 
\begin{equation}\label{Eq4.2}
U e_\gamma U^*
= 
\sum_{z\in\mathfrak{Z}_n}\sum_{\gamma_1,\gamma_2\in\Gamma} \int_G \langle \gamma^{-1}\gamma_1\gamma_2,g\rangle\,dg\,p_z(\gamma_1)\otimes e^{(00)}_{\gamma_2}
= 
\sum_{z\in\mathfrak{Z}_n}\sum_{\gamma' \in \Gamma} p_z(\gamma \gamma'{}^{-1})\otimes e^{(00)}_{\gamma'} 
\end{equation}
for every $\gamma \in \Gamma$. 

\begin{lemma}\label{L4.1} There is a unique bijective $*$-homomorphism 
\[
\Phi_n : \widetilde{A}_n \longrightarrow \bigoplus_{(z,\gamma) \in \mathfrak{Z}_n\times\Gamma} zA_n \quad \Bigg(\cong \bigoplus_{(z,\gamma) \in \mathfrak{Z}_n\times\Gamma} B(\mathcal{H}_z)\Bigg)
\]
such that
\begin{equation}\label{Eq4.3}
\Phi_n(\pi_{\alpha_n}(a))(z,\gamma') := za, \qquad 
\Phi_n(e_\gamma)(z,\gamma') := p_z(\gamma\gamma'{}^{-1}) 
\end{equation} 
hold for any $a \in A_n$, $z \in \mathfrak{Z}_n$ and $\gamma,\gamma' \in \Gamma$. The map $\Phi_n$ intertwines the dual action $\widetilde{\alpha}^\gamma$ with the translation action of $\Gamma$ on the right coordinate, that is, 
\begin{equation}\label{Eq4.4}
\Phi_n(\widetilde{\alpha}^\gamma(x))(z,\gamma') = \Phi_n(x)(z,\gamma^{-1}\gamma')
\end{equation}
holds for any $x \in \widetilde{A}_n$ and $z \in \mathfrak{Z}_n$ and $\gamma,\gamma' \in \Gamma$. 
\end{lemma}
\begin{proof}
Remark that $A_n\,\bar{\otimes}\,L(G) \cong A_n\,\bar{\otimes}\,\ell^\infty(\Gamma) \cong \bigoplus_{\gamma\in\Gamma}A_n \cong \bigoplus_{(z,\gamma)\in\mathfrak{Z}_n\times\Gamma} zA_n$ by 
\[
a\otimes\lambda_g \leftrightarrow \sum_{\gamma \in \Gamma} \langle\gamma,g\rangle\,a\otimes\delta_\gamma \leftrightarrow 
(\langle \gamma,g\rangle\,a)_{\gamma\in\Gamma} = (\langle \gamma,g\rangle\,za)_{(z,\gamma)\in\mathfrak{Z}_n\times\Gamma}, \qquad a \in A_n, \quad g \in G  
\] 
with $L(G) := \lambda(G)''$ on $L^2(G)$. Therefore, the composition of $\mathrm{Ad}U$ and this bijective $*$-homomorphism gives the desired $\Phi_n$. By \eqref{Eq4.1} and \eqref{Eq4.2}  we obtain 
\[
\Phi_n(e_\gamma)(z,\gamma') = (Ue_\gamma U^*)(z,\gamma') = p_z(\gamma\gamma'{}^{-1})
\]
for every $\gamma \in \Gamma$. Hence we have confirmed that \eqref{Eq4.3} actually holds true. Since the $e_\gamma$ are the spectral projections of $\lambda(g)$ ($g \in G$), it is clear that \eqref{Eq4.3} determines $\Phi_n$ completely. 

We have 
\begin{gather*}
\Phi_n(\widetilde{\alpha}_n^\gamma(\pi_{\alpha_n}(a)))(z,\gamma') = \Phi_n(\pi_{\alpha_n}(a))(z,\gamma') = za =  \Phi_n(\pi_{\alpha_n}(a))(z,\gamma^{-1}\gamma'), \\
\Phi_n(\widetilde{\alpha}_n^\gamma(e_{\gamma''}))(z,\gamma') = \Phi_n(e_{\gamma\gamma''})(z,\gamma') = p_z(\gamma\gamma''\gamma'{}^{-1}) 
= p_z(\gamma''(\gamma^{-1}\gamma')^{-1}) = \Phi_n(e_{\gamma''})(z,\gamma^{-1}\gamma') 
\end{gather*}
({\it n.b.,} $\Gamma$ is commutative). Hence \eqref{Eq4.4} holds true. 
\end{proof}

We then investigate the inclusion $\widetilde{A}_n \hookrightarrow \widetilde{A}_{n+1}$ in the description of Lemma \ref{L4.1}. Note that the lemma, in particular, says that the inductive sequence $\widetilde{A}_n$ consists of finite, atomic $W^*$-algebras again. 

\medskip
Since $\alpha_{n+1}^g= \alpha_n^g$ holds on $A_n$ for every $g \in G$ thanks to the density of $\mathbb{R}$ in $G$, we observe that $g \in G \mapsto w_{n+1,n}(g) := u_n(g)^* u_{n+1}(g) \in (A_n)'\cap A_{n+1}$ gives a unitary representation. Since all the $zz' \neq 0$ with $(z,z') \in \mathfrak{Z}_{n+1}\times\mathfrak{Z}_n$ form a complete set of minimal central projections of $(A_n)'\cap A_{n+1}$, we obtain a unitary representation 
\[
g \in G \mapsto w_{z,z'}(g) := zz' w_{n+1,n}(g) = u_{z}(g) u_{z'}(g)^* = u_{z'}(g)^* u_{z}(g) \in zz'((A_n)'\cap A_{n+1})
\] 
for each $(z,z') \in \mathfrak{Z}_{n+1}\times\mathfrak{Z}_n$ with $zz' \neq 0$. Since $w_{z,z'}(g)$ is a unitary representation of a compact abelian group, it admits a spectral decomposition of the following form: 
\begin{equation}\label{Eq4.5}
w_{(z,z')}(g) = \sum_{\gamma \in \Gamma} \langle \gamma,g\rangle\,q_{(z,z')}(\gamma), \qquad g \in G, 
\end{equation}
where the $q_{(z,z')}(\gamma)$ form a partition of unity of $zz'((A_n)'\cap A_{n+1})$ consisting of projections. Since $\alpha_{n+1}^t = \alpha_n^t$ holds on $A_n$ for every $t \in \mathbb{R}$, we see that $\rho_z\rho_{z'} = \rho_{z'}\rho_z$ holds in $zA_{n+1}$ for each $(z,z') \in \mathfrak{Z}_{n+1}\times\mathfrak{Z}_n$ with $zz'\neq0$. Hence the generator of $w_{z,z'}(t)$ should be $\rho_z\rho_{z'}^{-1}=\rho_{z'}^{-1}\rho_z$, and thus we have the following explicit description of $q_{(z,z')}(\gamma)$ in terms of $p_z(\gamma)$: 
\begin{equation}\label{Eq4.6}
q_{(z,z')}(\gamma) 
= 
\sum_{\gamma'\in\Gamma} p_z(\gamma\gamma')p_{z'}(\gamma')
= 
\sum_{\gamma'\in\Gamma} p_{z'}(\gamma')p_z(\gamma'\gamma), \qquad \gamma \in \Gamma. 
\end{equation} 

\medskip
We define an element $a \otimes \delta_\gamma \in \Phi_n(\widetilde{A}_n)$ with $a \in A_n$ and $\gamma \in \Gamma$ by 
\[
(a\otimes \delta_\gamma)(z',\gamma') := \delta_\gamma(\gamma')\,z' a, \qquad (z',\gamma') \in \mathfrak{Z}_n\times\Gamma, 
\]
where $\delta_\gamma$ denotes the Dirac function at $\gamma$. We remark that the $z\otimes\delta_\gamma$, $(z,\gamma) \in \mathfrak{Z}_n\times\Gamma$, form a complete set of minimal central projections of $\Phi_n(\widetilde{A}_n)$. 

\begin{lemma}\label{L4.2} The embedding $\iota_{n+1,n} = \Phi_{n+1}\circ\Phi_n^{-1} : \Phi_n(\widetilde{A}_n) \hookrightarrow \Phi_{n+1}(\widetilde{A}_{n+1})$ obtained from $\widetilde{A}_n \hookrightarrow \widetilde{A}_{n+1}$ sends each $z'\otimes\delta_{\gamma'}$ with $(z',\gamma') \in \mathfrak{Z}_n\times\Gamma$ to 
\begin{equation}\label{Eq4.7}
\iota_{n+1,n}(z'\otimes\delta_{\gamma'}) = \sum_{\substack{z \in \mathfrak{Z}_{n+1} \\ m(z,z') > 0}} \sum_{\gamma \in \Gamma} q_{(z,z')}(\gamma'\gamma^{-1})\otimes\delta_\gamma.  
\end{equation}
In particular,  
\begin{equation}\label{Eq4.8}
(z\otimes\delta_\gamma)\,\iota_{n+1,n}(z'\otimes\delta_{\gamma'}) = 
%\iota_{n+1,n}(z'\otimes\delta_{\gamma'})\,(z\otimes\delta_\gamma) =  
\begin{cases} 
q_{(z,z')}(\gamma'\gamma^{-1})\otimes\delta_\gamma & (m(z,z') > 0), \\
0 & (m(z,z') = 0) 
\end{cases}
\end{equation}
for each pair $((z,\gamma),(z',\gamma')) \in (\mathfrak{Z}_{n+1}\times\Gamma)\times(\mathfrak{Z}_n\times\Gamma)$. 
\end{lemma}
\begin{proof} 
Choose an arbitrary pair $(z',\gamma') \in \mathfrak{Z}_n\times\Gamma$. By the proof of Lemma \ref{L4.1} we have 
\[
\Phi_n\Big(\int_G \overline{\langle\gamma',g\rangle}\,\pi_{\alpha_n}(u_{z'}(g)^*)\,\lambda(g)\,dg\Big) = z'\otimes\delta_{\gamma'}. 
\]
Observe that  
\[
\xymatrix{
\int_G \overline{\langle\gamma',g\rangle}\,\pi_{\alpha_n}(u_{z'}(g)^*)\,\lambda(g)\,dg \ar@{=}[d] & \quad \text{in $\widetilde{A}_n$} \\
\int_G \overline{\langle\gamma',g\rangle}\,\pi_{\alpha_{n+1}}(u_{z'}(g)^*)\,\lambda(g)\,dg \ar@{<->}[d] & \qquad \text{in $\widetilde{A}_{n+1}$} \\
\int_G \overline{\langle\gamma',g\rangle}\,(u_{z'}(g)^* u_{n+1}(g))\otimes\lambda_g\,dg & \qquad \text{in $A_n\,\bar{\otimes}\,L(G)$}.
}
\]
We have, by \eqref{Eq4.5} and the proof of Lemma \ref{L3.4} (formula $\lambda_g e_\gamma^{(00)} = \langle\gamma,g\rangle\,e_\gamma^{(00)}$), 
\begin{align*}
\int_G \overline{\langle\gamma',g\rangle}\,(u_{z'}(g)^* u_{n+1}(g))\otimes\lambda_g\,dg 
&= 
\sum_{\substack{z \in \mathfrak{Z}_{n+1} \\ m(z,z') > 0}} 
\int_G \overline{\langle\gamma',g\rangle}\,w_{(z,z')}(g)\otimes\lambda_g\,dg \\
&= 
\sum_{\substack{z \in \mathfrak{Z}_{n+1} \\ m(z,z') > 0}} 
\sum_{\gamma_1, \gamma_2 \in \Gamma} 
\int_G \langle \gamma'{}^{-1} \gamma_1\gamma_2,g\rangle\,q_{(z,z')}(\gamma_1)\otimes e_{\gamma_2}^{(00)}\,dg \\
&= 
\sum_{\substack{z \in \mathfrak{Z}_{n+1} \\ m(z,z') > 0}} 
\sum_{\gamma \in \Gamma} q_{(z,z')}(\gamma'\gamma^{-1})\otimes e_\gamma^{(00)}. 
\end{align*}
It follows that 
\[
\Phi_{n+1}\Big(\int_G \overline{\langle\gamma',g\rangle}\,\pi_{\alpha_n}(u_{z'}(g)^*)\,\lambda(g)\,dg\Big) 
= 
\sum_{\substack{z \in \mathfrak{Z}_{n+1} \\ m(z,z') > 0}} 
\sum_{\gamma \in \Gamma} q_{(z,z')}(\gamma'\gamma^{-1})\otimes\delta_\gamma. 
\]
Consequently, we obtain \eqref{Eq4.7}, which trivially implies \eqref{Eq4.8}. 
\end{proof}

The lemmas above immediately imply the following proposition: 

\begin{proposition}\label{P4.3} The minimal central projections of $\widetilde{A}_n$ are labeled by $\widetilde{\mathfrak{Z}}_n := \mathfrak{Z}_n\times\Gamma$, and the dimension corresponding to a $(z,\gamma) \in \widetilde{\mathfrak{Z}}_n$ becomes $\dim(z)$ (i.e., being independent of $\gamma$). 

The branching graph $(\widetilde{\mathfrak{Z}},\tilde{m})$ of the inductive sequence $\widetilde{A}_n$ is given by $\widetilde{\mathfrak{Z}} := \bigsqcup_{n\geq0} \widetilde{\mathfrak{Z}}_n$ and
\begin{align*}
\tilde{m}((z,\gamma),(z',\gamma')) 
&= \frac{\mathrm{Tr}(\iota_{n+1,n}(z'\otimes\delta_{\gamma'})(z,\gamma))}{\dim(z')} \\
&= 
\begin{cases} 
\frac{\mathrm{Tr}(q_{(z,z')}(\gamma^{-1}\gamma'))}{\dim(z')} & (m(z,z') > 0), \\
0 & (m(z,z') = 0) 
\end{cases}
\end{align*}
for any $((z,\gamma),(z',\gamma')) \in \widetilde{\mathfrak{Z}}_{n+1}\times\widetilde{\mathfrak{Z}}_n$, $n \geq 0$. 
In particular, the standard link $\tilde{\mu}$ over $(\widetilde{\mathfrak{Z}},\tilde{m})$ becomes 
\begin{align*}
\tilde{\mu}((z,\gamma),(z',\gamma')) 
&=
\tilde{m}((z,\gamma),(z',\gamma'))\,\frac{\dim(z')}{\dim(z)} \\
&= 
\begin{cases} 
\frac{\mathrm{Tr}(q_{(z,z')}(\gamma^{-1}\gamma'))}{\dim(z)} & (m(z,z') > 0), \\
0 & (m(z,z') = 0) 
\end{cases}
\end{align*}
for any $((z,\gamma),(z',\gamma')) \in \widetilde{\mathfrak{Z}}_{n+1}\times\widetilde{\mathfrak{Z}}_n$, $n \geq 0$. 

In particular, the multiplicity function $\tilde{m}$ and the standard link $\tilde{\mu}$ are invariant under the translation action $T: \Gamma \curvearrowright \widetilde{\mathfrak{Z}}$ defined by $T_\gamma(z,\gamma') := (z,\gamma\gamma')$, that is, 
\[
\tilde{\mu}\circ(T_\gamma^{-1}\times T_\gamma^{-1}) = \tilde{\mu}, \quad \tilde{m}\circ(T_\gamma^{-1}\times T_\gamma^{-1}) = \tilde{m}, \qquad \gamma \in \Gamma. 
\] 
\end{proposition}

\begin{remark}\label{R4.4}{\rm 
Lemma \ref{L4.1} says that 
\[
\widetilde{A}_n \cong \Phi_n(\widetilde{A}_n) = \bigoplus_{(z,\gamma)\in\mathfrak{Z}_n\times\Gamma} \overset{\tiny z\otimes\delta_\gamma}{zA_n} \quad \text{with $zA_n = B(\mathcal{H}_z)$}, 
\]
where the symbol $z\otimes\delta_\gamma$ over $zA_n$ indicates the central support projection of direct summand $zA_n$. Then its center-valued trace $\mathrm{ctr}_n$ is given by 
\begin{equation}
\mathrm{ctr}_n(x)(z,\gamma) = \frac{\mathrm{Tr}(x(z,\gamma))}{\dim(z)}, \qquad x \in \Phi_n(\widetilde{A}_n),\quad (z,\gamma) \in \mathfrak{Z}_n\times\Gamma, 
\end{equation}
where $\mathrm{Tr}$ stands for the non-normalized trace on $zA_n = B(\mathcal{H}_z)$. (See \cite[Theorem V.2.6]{Takesaki:Book1}; its uniqueness guarantees that the above map is indeed the center-valued trace.) We observe that 
\begin{equation}\label{Eq4.10}
\mathrm{ctr}_{n+1}(\iota_{n+1,n}(z'\otimes\delta_{\gamma'}))(z,\gamma) = \tilde{\mu}((z,\gamma),(z',\gamma'))
\end{equation}
holds for every pair $((z,\gamma),(z',\gamma')) \in \widetilde{\mathfrak{Z}}_{n+1}\times\widetilde{\mathfrak{Z}}_n$, $n \geq 0$. This is consistent with \cite[equation (3.7)]{Ueda:SIGMA22}, and the natural conditional expectations playing the role of $E^{(\alpha_n^t,\beta)}$ in \cite{Ueda:SIGMA22} are the center-valued traces of $\widetilde{A}_n$ in the present context. }
\end{remark}

\subsection{Harmonic functions corresponding to $(\widetilde{\alpha}^\gamma,\beta)$-scaling traces}\label{S4.2} 

So far, we have described the branching graph $(\widetilde{\mathfrak{Z}},\tilde{m})$ associated with the $\widetilde{A}_n$, $n\geq0$. With the description, we will translate the $(\widetilde{\alpha}^\gamma,\beta)$-traces $TW_\beta^\mathrm{ln}(\widetilde{\alpha}^\gamma)$ into a certain class of harmonic functions on $(\widetilde{\mathfrak{Z}},\tilde{m})$. 

\begin{lemma}\label{L4.5} For each $\tau \in TW_\beta^\mathrm{ln}(\widetilde{\alpha}^\gamma)$ there is a unique function $\tilde{\nu} = \tilde{\nu}[\tau] : \widetilde{\mathfrak{Z}} := \bigsqcup_{n\geq0}\widetilde{\mathfrak{Z}}_n \to [0,+\infty)$ such that 
\begin{equation}\label{Eq4.11}
\tau(x) = \sum_{(z,\gamma)\in\widetilde{\mathfrak{Z}}_n} \tilde{\nu}(z,\gamma)\,\frac{\mathrm{Tr}(\Phi_n(x)(z,\gamma))}{\dim(z)}, \qquad x \in \widetilde{A}_n. 
\end{equation}
The function $\tilde{\nu}$ enjoys the following properties: 
\begin{itemize}
\item[(i)] $\displaystyle \tilde{\nu}(z',\gamma') = \sum_{(z,\gamma) \in \widetilde{\mathfrak{Z}}_{n+1}}\tilde{\nu}(z,\gamma)\,\tilde{\mu}((z,\gamma),(z',\gamma'))$ for all $(z',\gamma') \in \widetilde{\mathfrak{Z}}_n$, $n \geq 0$. 
\item[(ii)] $\displaystyle \tilde{\nu}(z,\gamma) = \gamma^\beta\tilde{\nu}(z,1)$ for all $(z,\gamma) \in \widetilde{\mathfrak{Z}}$. 
\item[(iii)] $\tilde{\nu}(1,1) = 1$. 
%$\displaystyle \sum_{z\in\mathfrak{Z}_n} \frac{\dim_\beta(z)}{\dim(z)}\,\tilde{\nu}(z,1) = 1$ for all $n \geq 0$. 
\end{itemize}
\end{lemma} 
\begin{proof} 
Write $\tau_n := \tau\circ\Phi_n^{-1}$ for simplicity, and it should be a normal semifinite tracial weight on $\Phi_n(\widetilde{A}_n)$. Since all the $z \otimes\delta_\gamma$ form a complete orthogonal family of minimal central projections of $\Phi_n(\widetilde{A}_n)$, we observe that $\tau_n(z\otimes\delta_\gamma) < +\infty$ for any $(z,\gamma) \in \widetilde{\mathfrak{Z}}_n$. Thus, 
\[
a \in \overset{z\otimes\delta_\gamma}{(zA_n)_+} \Big(\subset \Phi_n(\widetilde{A}_n)_+\Big) \mapsto \tau_n(a) 
\in [0,+\infty)
\]
(see Remark \ref{R4.4} for this notation of direct summands) coincides with a unique non-negative scalar multiple of the normalized trace $\mathrm{Tr}(\,\cdot\,)/\dim(z)$ on $zA_n = B(\mathcal{H}_z)$. Then, the non-negative scalar gives the desired number $\tilde{\nu}(z,\gamma)$, that is, by the semifiniteness and normality, we have 
\[
\tau_n(x) 
= 
\sum_{(z,\gamma) \in \widetilde{\mathfrak{Z}}_n} \tau_n((z\otimes\delta_\gamma)x) 
= 
\sum_{(z,\gamma) \in \widetilde{\mathfrak{Z}}_n} \tilde{\nu}(z,\gamma)\frac{\mathrm{Tr}(x(z,\gamma))}{\dim(z)}\ 
\Big(= \tau_n(\mathrm{ctr}_n(x)) \Big)
\]
for all $x \in \Phi_n(\widetilde{A}_n)_+$. Hence \eqref{Eq4.11} follows.  

\medskip
Item (i): We have
\begin{align*}
\tilde{\nu}(z',\gamma') 
&= 
\tau_n(z'\otimes\delta_{\gamma'}) \\
&= 
\tau_{n+1}(\iota_{n+1,n}(z'\otimes\delta_{\gamma'})) \\
&= 
\sum_{(z,\gamma) \in \widetilde{\mathfrak{Z}}_{n+1}} \tilde{\nu}(z,\gamma)\frac{\mathrm{Tr}(\iota_{n+1,n}(z'\otimes\delta_{\gamma'})(z,\gamma))}{\dim(z)} \\
&= 
\sum_{(z,\gamma) \in \widetilde{\mathfrak{Z}}_{n+1}} \tilde{\nu}(z,\gamma)\,\tilde{\mu}((z,\gamma),(z',\gamma'))
\end{align*}
by Proposition \ref{P4.3} (and Remark \ref{R4.4}).  

\medskip
Item (ii): We observe that 
\[
\Phi_n(\widetilde{\alpha}^\gamma(\Phi_n^{-1}(z\otimes\delta_1)))(z',\gamma') 
= 
\Phi_n(\Phi_n^{-1}(z\otimes\delta_1))(z',\gamma^{-1}\gamma') 
= 
(z\otimes\delta_\gamma)(z',\gamma')
\]
for $(z',\gamma') \in \widetilde{\mathfrak{Z}}_n$. Hence we have $\widetilde{\alpha}^\gamma(\Phi_n^{-1}(z\otimes\delta_1)) = \Phi_n^{-1}(z\otimes\delta_\gamma)$, and hence 
\[
\tilde{\nu}(z,\gamma) = \tau(\Phi_n^{-1}(z\otimes\delta_\gamma)) = \tau(\widetilde{\alpha}^\gamma(\Phi_n^{-1}(z\otimes\delta_1))) = 
\gamma^\beta\,\tau(\Phi_n^{-1}(z\otimes\delta_1)) = \gamma^\beta\,\tilde{\nu}(z,1) 
\]
by item (ii) of Definition \ref{D3.2}(1). 

\medskip 
Iten (iii): This is nothing but item (iii) of Definition \ref{D3.2}(1), i.e., $\tau(e_1) = 1$. 
%We have 
%\begin{align*}
%1
%&=
%\tau(e_1) 
%=
%\sum_{(z,\gamma)\in\widetilde{\mathfrak{Z}}_n} \tilde{\nu}(z,\gamma)\frac{\mathrm{Tr}(p_z(\gamma^{-1}))}{\dim(z)} \\
%&= 
%\sum_{z\in\mathfrak{Z}_n} \tilde{\nu}(z,1) \frac{1}{\dim(z)}\sum_{\gamma\in\Gamma} \gamma\,\mathrm{Tr}(p_z(\gamma^{-1})) 
%= 
%\sum_{z\in\mathfrak{Z}_n} \frac{\dim_\beta(z)}{\dim(z)}\,\tilde{\nu}(z,1), 
%\end{align*}
%where the first equality follows from (the right-hand side of) \eqref{Eq4.3}.  
\end{proof} 

We remark that 
\[
\sum_{z\in\mathfrak{Z}_n} \frac{\dim_\beta(z)}{\dim(z)}\,\tilde{\nu}(z,1) = 1, \qquad n \geq 0, 
\]
which follows from items (i)-(iii) above thanks to Proposition \ref{P4.3}. 

\begin{definition}\label{D4.6} A \emph{normalized, $\beta$-power scaling $\tilde{\mu}$-harmonic function} is a function $\tilde{\nu} : \widetilde{\mathfrak{Z}} \to [0,+\infty)$ such that items (i)-(iii) in Proposition \ref{L4.5} hold. We denote by $H_1^+(\tilde{\mu})_\beta$ all the normalized, $\beta$-power scaling $\tilde{\mu}$-harmonic functions. 
\end{definition}

We will use the notion of $\kappa$-harmonic functions and notation $H_1^+(\kappa)$ in \cite[section 7]{Ueda:preprint20} in what follows. 

\begin{theorem}\label{T4.7} There is a unique affine-isomorphism $\nu \in H_1^+(\kappa) \longleftrightarrow \tilde{\nu} \in H_1^+(\tilde{\mu})_\beta$ with 
\[
\dim_\beta(z)\,\tilde{\nu}(z,\gamma) = \dim(z)\,\nu(z)\,\gamma^\beta, \qquad (z,\gamma) \in \widetilde{\mathfrak{Z}}. 
\]
\end{theorem} 
\begin{proof} 
We first claim that 
\begin{equation}\label{Eq4.12}
\frac{\mathrm{Tr}(\rho_z^{-\beta}x)}{\mathrm{Tr}(\rho_z^{-\beta}z')} 
= 
\frac{\mathrm{Tr}(\rho_{z'}^{-\beta}x)}{\dim_\beta(z')}, \qquad x \in z'A_n = B(\mathcal{H}_{z'}) \hookrightarrow zA_{n+1} = B(\mathcal{H}_z)
\end{equation}
holds for any pair $(z,z') \in \mathfrak{Z}_{n+1}\times\mathfrak{Z}_n$ with $m(z,z') > 0$. In fact, the left-hand side defines an $(\alpha_{z'}^t,\beta)$-KMS state on $z'A_n = B(\mathcal{H}_{z'})$, and the uniqueness of $(\alpha_{z'}^t,\beta)$-KMS state shows the claim. 

\medskip
Let $\nu \in H_1^+(\kappa)$ be arbitrarily chosen, and we will show that 
\[
\tilde{\nu}(z,\gamma) := \frac{\dim(z)}{\dim_\beta(z)}\,\nu(z)\,\gamma^\beta
\]
defines an element of $H_1^+(\tilde{\mu})_\beta$. Item (ii) of Lemma \ref{L4.5} trivially holds, and the normalization property of $\nu$ trivially implies item (iii) of Lemma \ref{L4.5}. Hence, it suffices to show item (i) of Lemma \ref{L4.5}. 

We have 
\begin{align*} 
\sum_{(z,\gamma) \in \widetilde{\mathfrak{Z}}_{n+1}} \tilde{\nu}(z,\gamma)\,\tilde{\mu}((z,\gamma),(z',\gamma)) 
&=
\sum_{\substack{(z,\gamma) \in \widetilde{\mathfrak{Z}}_{n+1} \\ m(z,z')>0}} \frac{\dim(z)}{\dim_\beta(z)}\,\nu(z)\,\gamma^\beta\,\frac{\mathrm{Tr}(q_{(z,z')}(\gamma^{-1}\gamma'))}{\dim(z)} \\
&= 
\frac{1}{\dim_\beta(z)}\sum_{\substack{z \in \mathfrak{Z}_{n+1} \\ m(z,z')>0}}  \nu(z)\sum_{\gamma \in \Gamma}\gamma^\beta\,\mathrm{Tr}(q_{(z,z')}(\gamma^{-1}\gamma')).  
\end{align*}
Now, we observe that 
\begin{align*}
\sum_{\gamma \in \Gamma}\gamma^\beta\,\mathrm{Tr}(q_{(z,z')}(\gamma^{-1}\gamma'))
&=
\sum_{\gamma \in \Gamma}\gamma^\beta \sum_{\gamma'' \in \Gamma} \mathrm{Tr}(p_z(\gamma^{-1}\gamma'\gamma'')p_{z'}(\gamma'')) \\
&= 
\gamma'{}^\beta\sum_{\gamma_1,\gamma_2 \in \Gamma} \gamma_1^{-\beta}\gamma_2^\beta\mathrm{Tr}(p_z(\gamma_1)p_{z'}(\gamma_2)) \\
&= 
\gamma'{}^\beta\,\mathrm{Tr}(\rho_z^{-\beta}\rho_{z'}^\beta) \\
&= 
\dim_\beta(z)\,\tau_z^\beta(zz')\frac{\dim(z')\,\gamma'{}^\beta}{\dim_\beta(z')}
\end{align*}
by \eqref{Eq4.12}. Since $\kappa(z,z') = \tau_z^\beta(zz')$ and since $zz' = 0$ if and only if $m(z,z') = 0$, we conclude that 
\begin{align*}
\sum_{(z,\gamma) \in \widetilde{\mathfrak{Z}}_{n+1}} \tilde{\nu}(z,\gamma)\,\tilde{\mu}((z,\gamma),(z',\gamma)) 
&=
\frac{\dim(z')\,\gamma'{}^\beta}{\dim_\beta(z')}\sum_{z \in \mathfrak{Z}_{n+1}} \nu(z)\,\tau_z^\beta(zz') \\
&=
\frac{\dim(z')\,\gamma'{}^\beta}{\dim_\beta(z')}\nu(z') = \tilde{\nu}(z',\gamma'). 
\end{align*}
Hence $\tilde{\nu}$ enjoys item (i) 0f Lemma \ref{L4.5}. 

\medskip
Let $\tilde{\nu} \in H_1^+(\tilde{\mu})_\beta$ be arbitrarily chosen, and we will show that 
\[
\nu(z) := \frac{\dim_\beta(z)}{\dim(z)}\tilde{\nu}(z,1)
\]
defines an element of $H_1^+(\kappa)$. 

We first observe that 
\begin{align*} 
\sum_{z\in\mathfrak{Z}_{n+1}} \nu(z)\kappa(z,z') 
&= 
\sum_{z\in\mathfrak{Z}_{n+1}} \tilde{\nu}(z,1) \frac{\mathrm{Tr}(\rho_z^{-\beta}z')}{\dim(z)} \\
&=
\sum_{z\in\mathfrak{Z}_{n+1}} \tilde{\nu}(z,1) \frac{\mathrm{Tr}(\rho_z^{-\beta}\rho_{z'}^\beta)\dim_\beta(z')}{\dim(z)\,\mathrm{Tr}(\rho_{z'}^{-\beta}\rho_{z'}^\beta)} \qquad \text{(use \eqref{Eq4.12})} \\
&=
\sum_{z\in\mathfrak{Z}_{n+1}} \tilde{\nu}(z,1) \sum_{\gamma \in \Gamma} \gamma^{-\beta}\frac{\mathrm{Tr}(q_{(z,z')}(\gamma))}{\dim(z)}\frac{\dim_\beta(z')}{\dim(z')} \\
&=
\frac{\dim_\beta(z')}{\dim(z')}\sum_{z\in\mathfrak{Z}_{n+1}} \tilde{\nu}(z,1) \sum_{\gamma \in \Gamma} \gamma^\beta\,\frac{\mathrm{Tr}(q_{(z,z')}(\gamma^{-1}))}{\dim(z)} \\
&=
\frac{\dim_\beta(z')}{\dim(z')}\sum_{(z,\gamma) \in \widetilde{\mathfrak{Z}}_{n+1}} \tilde{\nu}(z,\gamma)\tilde{\mu}((z,\gamma),(z',1)) \quad \text{(by Proposition \ref{P4.3})}\\
&=
\frac{\dim_\beta(z')}{\dim(z')}\tilde{\nu}(z',1) = \nu(z'). 
\end{align*}
Hence $\nu$ is $\kappa$-harmonic. Moreover, item (iii) of Lemma \ref{L4.5}, a requirement of $\tilde{\nu}$, clearly shows that $\nu$ is normalized. Hence we are done. 
\end{proof} 

So far, we have obtained the following diagram: 
\[
\xymatrix{ 
K_\beta^\mathrm{ln}(\alpha^t) \ar@{<->}[r]^{(a)\ } \ar@{<->}[d]_{(b)} & TW_\beta^\mathrm{ln}(\widetilde{\alpha}^\gamma) \ar[d]^{(c)} \\
H_1^+(\kappa) \ar@{<->}[r]_{(d)\ } & H_1^+(\tilde{\mu})_\beta,   
}
\]
where the correspondences (a)-(d) have been established as follows. 
\begin{itemize}
\item[(a)] Theorem \ref{T3.7}; 
\item[(b)] \cite[Proposition 3.7]{Ueda:preprint20}; 
\item[(c)] Lemma \ref{L4.5}; 
\item[(d)] Theorem \ref{T4.7}. 
\end{itemize}
We will examine the composition of maps (d) $\to$ (b) $\to$ (a).  

Let $\tilde{\nu} \in H_1^+(\tilde{\mu},\beta)$ be arbitrarily chosen. By Theorem \ref{T4.7} we have a unique $\nu \in H_1^+(\kappa)$ with 
\[
\nu(z) = \frac{\dim_\beta(z)}{\dim(z)}\tilde{\nu}(z,1), \qquad z \in \mathfrak{Z}.
\] 
Then, by \cite[Proposition 3.7]{Ueda:preprint20} we have a unique $\omega \in K_\beta^\mathrm{ln}(\alpha^t)$ so that 
\[
\omega(a) = \sum_{z \in \mathfrak{Z}_n} \nu(z)\,\tau_z^\beta(za)=\sum_{z \in \mathfrak{Z}_n}\frac{\dim_\beta(z)}{\dim(z)}\tilde{\nu}(z,1)\,\tau_z^\beta(za), \qquad a \in A_n,\ n\geq0.
\] 
Finally, with this $\omega$ we obtain a unique $\tau_\omega = \mathrm{tr}_\beta\circ \widetilde{E}_\omega \in TW_\beta^\mathrm{ln}(\widetilde{\alpha}^\gamma)$ by Theorem \ref{T3.7}. Consequently, the resulting $\tau_\omega$ enjoys 
\[
\tilde{\nu}[\tau_\omega](z,\gamma) = \tau_\omega(\Phi_n^{-1}(z\otimes\delta_\gamma)) = \mathrm{tr}_\beta(E_\omega(\Phi_n^{-1}(z\otimes\delta_\gamma))). 
\]
By the proof of Lemma \ref{L4.2} we observe that 
\[
\Phi_n^{-1}(z\otimes\delta_\gamma) 
= 
\int_G \overline{\langle\gamma,g\rangle}\,\pi_{\alpha_n}(u_z(g)^*)\,\lambda(g)\,dg \\
= 
\sum_{\gamma_1^{-1}\gamma_2 = \gamma} \pi_{\alpha_n}(p_z(\gamma_1)) e_{\gamma_2}. 
\]
Consequently, we obtain that 
\begin{align*}
\tilde{\nu}[\tau_\omega](z,\gamma) 
&= 
\sum_{\gamma_1^{-1}\gamma_2=\gamma} \dim_\beta(z)\,\tilde{\nu}(z,1)\frac{1}{\dim(z)} \tau_z^\beta(p_z(\gamma_1))\,\gamma_2^\beta \\
%&= 
%\sum_{\gamma_1^{-1}\gamma_2=\gamma} \dim_\beta(z)\,\tilde{\nu}(z,1)\frac{1}{\dim(z)}\,\frac{\mathrm{Tr}(\rho_z^{-\beta}p_z(\gamma_1))}{\dim_\beta(z)}\,\gamma_2^\beta \\
&= 
\sum_{\gamma_1^{-1}\gamma_2=\gamma} \dim_\beta(z)\,\tilde{\nu}(z,1)\frac{1}{\dim(z)}\,\frac{\gamma_1^{-\beta}\mathrm{Tr}(p_z(\gamma_1))}{\dim_\beta(z)}\,\gamma_2^\beta \\
&=
\tilde{\nu}(z,1)\,\gamma^\beta\,\sum_{\gamma_1} \frac{\mathrm{Tr}(p_z(\gamma_1))}{\dim(z)} \\
&= \tilde{\nu}(z,1)\gamma^\beta 
= \tilde{\nu}(z,\gamma). 
\end{align*}
It follows that the composition of maps (d) $\to$ (b) $\to$ (a) is exactly inverse to map (c). Hence we have arrived at the following theorem: 

\begin{theorem}\label{T4.8} The mapping $\tau \in TW_\beta^\mathrm{ln}(\widetilde{\alpha}^\gamma) \mapsto \tilde{\nu}[\tau] \in H_1^+(\tilde{\mu})_\beta$ obtained in Lemma \ref{L4.5} is an affine-isomorphism. 
\end{theorem} 

\subsection{Weights and Weight-extended branching graph of link}\label{S4.3} 

The reader might ask how to construct the branching graph $(\widetilde{\mathfrak{Z}},\tilde{m})$ with $\Gamma$-action from a given link $(\mathfrak{Z},\kappa)$ rather than an inductive $C^*$-flow $\alpha^t$. Such a construction can be given by using \cite[section 9]{Ueda:preprint20}; namely, one first constructs an inductive $C^*$-flow from $(\mathfrak{Z},\kappa)$, and then apply the discussions so far in this paper to it. Here we will translate this procedure without appealing to any $C^*$-flows. This seems to be of independent interest. 

\medskip 
We first remark that the analysis of links does not depend on multiplicities on edges; hence we will ignore, for simplicity, the multiplicity function over $\mathfrak{Z}$. Here one should remark that $m(z,z') > 0$ if and only if $\kappa(z,z')>0$, and hence the edges $(z,z') \in \bigsqcup_{n\geq0}\mathfrak{Z}_{n+1}\times\mathfrak{Z}_n$ are determined by the positivity of $\kappa(z,z')$. Moreover, we have assumed that 
\[
\bigcup_{z' \in \mathfrak{Z}_n}\{z \in \mathfrak{Z}_{n+1}; \kappa(z,z') > 0\} = \mathfrak{Z}_{n+1}, \qquad \bigcup_{z \in \mathfrak{Z}_{n+1}}\{z' \in \mathfrak{Z}_n; \kappa(z,z') > 0\} = \mathfrak{Z}_n
\] 
for all $n\geq0$. (Informally, this assumption corresponds to that $A_n \hookrightarrow A_{n+1}$ is a unital embedding for every $n\geq0$.) We will assume that our link satisifies these requirements. 

Since the definition of $\kappa$ in \eqref{Eq2.3} involves the inverse temperature $\beta$, we have to specify this $\beta$. In what follows, we informally think that the inverse temperature has been selected to be $\beta=-1$. 

\begin{definition}\label{D4.9} For each $z \in \mathfrak{Z}_n$, $n \geq 0$ we define its \emph{$\kappa$-dimension} by   
\begin{equation}\label{Eq4.13} 
\kappa\mathchar`-\dim(z) := \sqrt{\sum_{\substack{z_k \in \mathfrak{Z}_k (k=0,1\dots,n) \\ z_0=1, z_n=z \\ \kappa(z_{k+1},z_k) > 0\, (k=0,1,\dots,n-1)}} \frac{1}{\kappa(z_n,z_{n-1})\kappa(z_{n-1},z_{n-2})\cdots\kappa(z_1,z_0)}}    
\end{equation}
with $\kappa\mathchar`-\dim(1) := 1$. We then define the \emph{weight} at $(z,z') \in \mathfrak{Z}_{n+1}\times\mathfrak{Z}_n$, $n \geq 0$ by 
\begin{equation}\label{Eq4.14}
\rho(z,z') :=  \kappa\mathchar`-\dim(z)\,\kappa(z,z')\frac{1}{\kappa\mathchar`-\dim(z')}.
\end{equation}
The countable discrete subgroup $\Gamma(\kappa)$ of $\mathbb{R}_+^\times$ generated by all the positive weights $\rho(z,z') >0$ with $(z,z') \in \mathfrak{Z}_{n+1}\times\mathfrak{Z}_n$, $n\geq0$ is called the \emph{weight group} of $\kappa$. 
\end{definition}

By definition, $\rho(z,z') > 0$ if and only if $\kappa(z,z') = 1$. This construction is motivated from that in \cite[Proposition 9.5]{Ueda:preprint20} together with \eqref{Eq4.5},\eqref{Eq4.6}. 

Here is a claim, which informally corresponds to $\mathrm{Tr}(\rho_z) = \mathrm{Tr}(\rho_z^{-1})$. 

\begin{lemma}\label{L4.9} We have
\begin{align*}
\kappa\mathchar`-\dim(z) 
&= 
\sum_{\substack{z_k \in \mathfrak{Z}_k (k=0,1\dots,n) \\ z_0=1, z_n=z}} \rho(z_n,z_{n-1})\rho(z_{n-1},z_{n-2})\cdots\rho(z_1,z_0) \\
&= 
\sum_{\substack{z_k \in \mathfrak{Z}_k (k=0,1\dots,n) \\ z_0=1, z_n=z \\ \kappa(z_{k+1},z_k)>0\,(k=0,1\dots,n-1)}} \frac{1}{\rho(z_n,z_{n-1})\rho(z_{n-1},z_{n-2})\cdots\rho(z_1,z_0)}
\end{align*}
for every $z \in \mathfrak{Z}_n$, $n\geq1$. 
\end{lemma}
\begin{proof} 
This is easily shown by induction on $n$. Clearly, $\kappa\mathchar`-\dim(z) = \kappa(z,1) = \rho(z,1) = 1$ holds for every $z \in \mathfrak{Z}_1$. The induction procedure from $n$ to $n+1$ goes as follows. Using $\sum_{z' \in \mathfrak{Z}_n} \kappa(z,z') = 1$ for every $z \in \mathfrak{Z}_{n+1}$, a property of links, we easily see that the first identity holds true. Compute 
\begin{align*} 
&\sum_{\substack{z_k \in \mathfrak{Z}_k (k=0,1\dots,n+1) \\ z_0=1, z_{n+1}=z \\ \kappa(z_{k+1},z_k)>0\,(k=0,1\dots,n)}} \frac{1}{\rho(z_{n+1},z_n)\rho(z_n,z_{n-1})\cdots\rho(z_1,z_0)} \\
&\quad=
\frac{1}{\kappa\mathchar`-\dim(z)}\sum_{\substack{ z_n \in \mathfrak{Z}_n \\ \kappa(z,z_n) > 0}} \frac{\kappa\mathchar`-\dim(z_n)}{\kappa(z,z_n)} \sum_{\substack{z_k \in \mathfrak{Z}_k (k=0,1\dots,n-1) \\ z_0=1 \\ \kappa(z_{k+1},z_k)>0\, (k=0,1\dots,n-1)}} \frac{1}{\rho(z_n,z_{n-1})\cdots\rho(z_1,z_0)} \\
&= 
\frac{1}{\kappa\mathchar`-\dim(z)}\sum_{\substack{ z_n \in \mathfrak{Z}_n \\ \kappa(z,z_n) > 0}} \frac{\kappa\mathchar`-\dim(z_n)^2}{\kappa(z,z_n)} \qquad\qquad \text{(by induction hypothesis)} \\
&\quad=
\frac{1}{\kappa\mathchar`-\dim(z)}\sum_{\substack{ z_n \in \mathfrak{Z}_n \\ \kappa(z,z_n) > 0}} \frac{1}{\kappa(z,z_n)} \sum_{\substack{z_k \in \mathfrak{Z}_k (k=0,1\dots,n-1) \\ z_0=1 \\\kappa(z_{k+1},z_k)>0\, (k=0,1\dots,n-1)}} \frac{1}{\kappa(z_n,z_{n-1})\cdots\kappa(z_1,z_0)} \\
&\quad= 
\frac{1}{\kappa\mathchar`-\dim(z)} \sum_{\substack{z_k \in \mathfrak{Z}_k (k=0,1\dots,n+1) \\ z_0=1, z_{n+1}=z \\ \kappa(z_{k+1},z_k)>0\, (k=0,1\dots,n)}} \frac{1}{\kappa(z_{n+1},z_n)\kappa(z_n,z_{n-1})\cdots\kappa(z_1,z_0)} \\
&\quad=
\kappa\mathchar`-\dim(z).  
\end{align*}
Hence we are done.  
\end{proof}

Proposition \ref{P4.3} suggests us to define the desired new branching graph as follows. 

\begin{definition}\label{D4.xx} The \emph{weight-extended branching graph} $(\widetilde{\mathfrak{Z}},\tilde{m})$ of $\kappa$ is defined to be $\widetilde{\mathfrak{Z}} = \bigsqcup_{n\geq0} \widetilde{\mathfrak{Z}}_n$ with $\widetilde{\mathfrak{Z}}_n := \mathfrak{Z}_n\times\Gamma$ and   
\begin{equation}\label{Eq4.15} 
\tilde{m}((z,\gamma),(z',\gamma')) := 
\begin{cases} 
1 & (\gamma^{-1}\gamma' = \rho(z,z')>0), \\
0 & (\text{otherwise}). 
\end{cases}
\end{equation}
\end{definition}

This multiplicity function $\tilde{m}$ is invariant under the translation action of $\Gamma$ on the right coordinate, that is, $\tilde{m}\circ(T_\gamma^{-1}\times T_\gamma^{-1}) = \tilde{m}$ for every $\gamma\in\Gamma$. 

\medskip
Since we have implicitly assumed that all $m(z,z')$ are either $0$ or $1$, the \emph{dimension} $\dim(z)$ of $z \in \mathfrak{Z}_n \subset \mathfrak{Z}$ in this context should be the total number of paths $(z_n,\dots,z_1,1)$ with $z_k \in \mathfrak{Z}_k$ and $\kappa(z_{k+1},z_k) > 0$. The \emph{dimension} $\dim(z,\gamma)$ of $(z,\gamma) \in \widetilde{\mathfrak{Z}}_n$ is defined to be the total number of paths ending at $(z,\gamma)$ and starting in the $0$th stage $\widetilde{\mathfrak{Z}}_0$ (which is no longer singleton). Here is a lemma. 

\begin{lemma}\label{L4.10} 
$\dim(z,\gamma) = \dim(z)$ always holds. 
\end{lemma}
\begin{proof}
Let $((z_n,\gamma_n),\dots,(z_1,\gamma_1),(1,\gamma_0))$ be a path in $\widetilde{\mathfrak{Z}}$ ending at $(z_n,\gamma_n)$ and starting in $\widetilde{\mathfrak{Z}}_0$. Then, $m(z_{k+1},z_k) =1$ holds for every $k=0,\dots,n-1$ with $z_0:=1$. Moreover, $\gamma_1 = \gamma_0/\rho(z_1,1)$, $\gamma_2 = \gamma_1/\rho(z_2,z_1) = \gamma_0/\rho(z_2,z_1)\rho(z_1,1)$, \dots, $\gamma_n = \gamma_0/\rho(z_n,z_{n-1})\cdots\rho(z_1,1)$ should hold. This means that each path is uniquely determined by the path $(z_n,\dots,z_1,1)$ in $\mathfrak{Z}$ and the relation $\gamma_0=\gamma_n\rho(z_n,z_{n-1})\cdots\rho(z_1,1)$. Hence,  the desired assertion must hold.  
\end{proof}

This lemma shows that the standard link $\tilde{\mu}$ over $(\widetilde{\mathfrak{Z}},\tilde{m})$ should be 
\begin{equation}\label{Eq4.16}
\tilde{\mu}((z,\gamma),(z',\gamma')) 
= 
\begin{cases}
\frac{\dim(z')}{\dim(z)} & (\tilde{m}((z,\gamma),(z',\gamma'))=1), \\
0 & (\text{otherwise}). 
\end{cases}
\end{equation}  
With the preparation so far, Theorem \ref{T4.7} actually holds as it is with $\beta=-1$ and $\dim_\beta(z) = \kappa\mathchar`-\dim(z)$ in the present setup. Its proof is an easy exercise now.

\section{Relation to $K_0$-groups}\label{S5} 

$K_0$-groups or dimension groups play a role of representation rings in asymptotic representation theory, but they do not fit spherical representations for $C^*$-flows (nor general links). Thus, we introduced, in our previous paper \cite{Ueda:SIGMA22}, a certain replacement of $K_0$-group by means of operator systems to investigate inductive $C^*$-flows. Here, we will give a way to connect the locally normal $(\alpha^t,\beta)$-KMS states $K_\beta^\mathrm{ln}(\alpha^t)$ to $K$-theory of $\rho$-extension $(\widetilde{\alpha} : \Gamma \curvearrowright \widetilde{A} = \varinjlim\widetilde{A}_n)$ under the assumption that all $\dim(z) < +\infty$. 

\medskip
We will investigate the $K_0$-group $K_0(\widetilde{A})$ and its positive cone $K_0(\widetilde{A})_+$ of $\widetilde{A}=\varinjlim \widetilde{A}_n$. By a standard fact on $K$-theory (see e.g., \cite[Proposition 8.1]{Effros:CBMSBook}) we have $K_0(\widetilde{A}) = \varinjlim K_0(\widetilde{A}_n)$ and $K_0(\widetilde{A}_n)_+ =\varinjlim K_0(\widetilde{A}_n)_+$. Thus, we first have to calculate each pair $K_0(\widetilde{A}_n)_+ \subset K_0(\widetilde{A}_n)$ and then have to do each embedding $K_0(\widetilde{A}_n) \hookrightarrow K_0(\widetilde{A}_{n+1})$. 

The first task was completed by just using \cite[Proposition 6.1]{Reich:K-theory01} as follows. It is convenient to transform each $\widetilde{A}_n$ to 
\[
\Phi_n(\widetilde{A}_n) = \bigoplus_{(z,\gamma)\in\mathfrak{Z}_n} \overset{z\otimes\delta_\gamma}{B(\mathcal{H}_z)}
\]
by Lemma \ref{L4.1} with notation in Remark \ref{R4.4}. By \cite[Proposition 6.1(iv)]{Reich:K-theory01} the $K_0$-group $K_0(\Phi_n(\widetilde{A}_n))$ is isomorphic, by the dimension function $\mathrm{cdim}_n := \dim_{\mathcal{Z}(\Phi_n(\widetilde{A}_n))}$ induced from the center-valued trace $\mathrm{ctr}_n$ ($=\mathrm{tr}_{\mathcal{Z}(\Phi_n(\widetilde{A}_n))}$ in \cite{Reich:K-theory01}), to
\begin{equation}\label{Eq5.1}
\begin{aligned}
&\tilde{\prod_{(z,\gamma)\in \widetilde{\mathfrak{Z}}_n}} \frac{\mathbb{Z}}{\dim(z)} \\
&\qquad:= 
\Bigg\{ f : \widetilde{\mathfrak{Z}}_n \to \mathbb{Q}\ ;\ \text{$\displaystyle f(z,\gamma) \in \frac{\mathbb{Z}}{\dim(z)}$ for each $(z,\gamma) \in \widetilde{\mathfrak{Z}}_n$ and $\displaystyle \sup_{(z,\gamma)\in\widetilde{\mathfrak{Z}}_n} |f(z,\gamma)| < +\infty$} \Bigg\},  
\end{aligned}
\end{equation} 
which sits in the center $\ell^\infty(\widetilde{\mathfrak{Z}}_n)=\mathcal{Z}(\Phi_n(\widetilde{A}_n))$. Here this identification of the center is given by $\delta_{(z,\gamma)} = z\otimes\delta_\gamma$. We will take a closer look at $\mathrm{cdim}_n$. In this case, the $K_0$-group is the Grothendieck group of the Murray--von Neumann equivalence classes $[P]_n$ of projections in $\mathbb{M}_\infty(\Phi_n(\widetilde{A}_n)) := \bigcup_{m\geq1} M_m(\mathbb{C})\otimes\Phi_n(\widetilde{A}_n)$, where the embedding $M_m(\mathbb{C})\otimes\Phi_n(\widetilde{A}_n) \hookrightarrow M_{m+1}(\mathbb{C})\otimes\Phi_n(\widetilde{A}_n)$ is the upper corner one. The addition (semigroup operation) on it is given by 
\[
[P]_n+[Q]_n := \Bigg[ \begin{bmatrix} P \\ & Q \end{bmatrix} \Bigg]_n.
\] 
Then, the mapping $[P]_n \mapsto (\mathrm{Tr}\otimes\mathrm{ctr}_n)(P)$ is well defined because $\mathrm{Tr}$ is the non-normalized trace. This mapping is nothing less than the dimension function $\mathrm{cdim}_n$. The commutative diagram in \cite[Proposition 6.1(ii)]{Reich:K-theory01} and the finiteness of the $W^*$-algebra in question show that the order arising from the positive cone $K_0(\Phi_n(\widetilde{A}_n))_+$ is the natural, point-wise one on $\ell^\infty(\widetilde{\mathfrak{Z}}_n)$. Hence $K_0(\Phi_n(\widetilde{A}_n))_+$ ($\subset K_0(\Phi_n(\widetilde{A}_n))$) is isomorphic via $\mathrm{cdim}_n$ to 
\begin{equation}\label{Eq5.2}
\Bigg[\tilde{\prod_{(z,\gamma)\in \widetilde{\mathfrak{Z}}_n}} \frac{\mathbb{Z}}{\dim(z)}\Bigg]_+ := \Bigg\{ f \in \tilde{\prod_{(z,\gamma)\in \widetilde{\mathfrak{Z}}_n}} \frac{\mathbb{Z}}{\dim(z)} \quad ;\quad \text{$f(z,\gamma) \geq 0$ for each $(z,\gamma) \in \widetilde{\mathfrak{Z}}_n$} \Bigg\}. 
\end{equation} 

We then investigate the embedding $K_0(\Phi_n(\widetilde{A}_n)) \hookrightarrow K_0(\Phi_n(\widetilde{A}_{n+1}))$ in description \eqref{Eq5.1}. The embedding is $\iota_{n+1,n}^*$ with $\iota_{n+1,n} = \Phi_{n+1}\circ\Phi_n^{-1}$ in Lemma \ref{L4.2}. Hence we need to compute 
\[
\iota_{n+1,n}^{**} := \mathrm{cdim}_{n+1}\circ\iota_{n+1,n}^*\circ(\mathrm{cdim}_{n})^{-1}
: 
\tilde{\prod_{(z,\gamma)\in \widetilde{\mathfrak{Z}}_n}} \frac{\mathbb{Z}}{\dim(z)} \to 
\tilde{\prod_{(z,\gamma)\in \widetilde{\mathfrak{Z}}_{n+1}}} \frac{\mathbb{Z}}{\dim(z)}.
\]
Since $\iota_{n+1,n}$ and the center-valued traces are normal, it suffices to compute this map against each $(1/\dim(z))\delta_{(z,\gamma)} = (1/\dim(z))\,z\otimes\delta_\gamma$ with $(z,\gamma)\in \widetilde{\mathfrak{Z}}_n$. Moreover, we observe that  
\[
\mathrm{cdim}_{n+1}\circ\iota_{n+1,n}^*\circ(\mathrm{cdim}_{n})^{-1}(z\otimes\delta_{\gamma})
=
\dim(z)\times
\mathrm{cdim}_{n+1}\circ\iota_{n+1,n}^*\circ(\mathrm{cdim}_{n})^{-1}\Big(\frac{1}{\dim(z)}z\otimes\delta_{\gamma}\Big)
\]
and hence 
\begin{align*}
\mathrm{cdim}_{n+1}\circ\iota_{n+1,n}^*\circ(\mathrm{cdim}_{n})^{-1}\Big(\frac{1}{\dim(z)}z\otimes\delta_{\gamma}\Big) 
&=
\frac{1}{\dim(z)}\,
\mathrm{cdim}_{n+1}\circ\iota_{n+1,n}^*\circ(\mathrm{cdim}_{n})^{-1}(z\otimes\delta_{\gamma}) \\
&=
\frac{1}{\dim(z)}\,
\mathrm{cdim}_{n+1}\circ\iota_{n+1,n}^*([z\otimes\delta_\gamma]_n) \\
&=
\frac{1}{\dim(z)}\,
\mathrm{cdim}_{n+1}([\iota_{n+1,n}(z\otimes\delta_\gamma)]_n) \\
&=
\frac{1}{\dim(z)}\,
\mathrm{ctr}_{n+1}(\iota_{n+1,n}(z\otimes\delta_\gamma)) \\
&=
\mathrm{ctr}_{n+1}\Big(\iota_{n+1,n}\Big(\frac{1}{\dim(z)}\, z\otimes\delta_\gamma\Big)\Big). 
\end{align*}
Therefore, we conclude that the desired embedding map $\iota_{n+1,n}^{**}$ is just the restriction of the normal map $\mathrm{ctr}_{n+1}\circ\iota_{n+1,n} : \mathcal{Z}(\Phi_n(\widetilde{A}_n)) \to \mathcal{Z}(\Phi_{n+1}(\widetilde{A}_{n+1}))$ to the range of $\mathrm{cdim}_n(K_0(\Phi_n(\widetilde{A}_n)))$. Actually, for an $f \in \tilde{\prod}_{(z,\gamma)\in \widetilde{\mathfrak{Z}}_n} \frac{\mathbb{Z}}{\dim(z)}$, we have 
\begin{equation}\label{Eq5.3}
\begin{aligned}
\iota_{n+1,n}^{**}(f)(z,\gamma) 
&=
\sum_{(z',\gamma')\in\widetilde{\mathfrak{Z}}_n}f(z',\gamma')\,\mathrm{ctr}_{n+1}(\iota_{n+1,n}(z'\otimes\delta_{\gamma'}))(z,\gamma) \\
&= 
\sum_{(z',\gamma')\in\widetilde{\mathfrak{Z}}_n}\tilde{\mu}((z,\gamma),(z',\gamma'))\,f(z',\gamma')
\end{aligned}
\end{equation} 
by \eqref{Eq4.10}. This computation shows that the embedding $\iota_{n+1,n}^{**}$ is the left-multiplication of $\infty\times\infty$ matrix 
\[
\begin{bmatrix} \tilde{\mu}((z,\gamma),(z',\gamma'))\end{bmatrix}_{\widetilde{\mathfrak{Z}}_{n+1}\times\widetilde{\mathfrak{Z}}_n}
\] 
in description \eqref{Eq5.1}. Since $\tilde{\mu}((z,\gamma),(z',\gamma')) \geq 0$, the embedding preserves the positivity. Summing up the discussions so far we conclude: 

\begin{proposition}\label{P5.1} The $(K_0(\widetilde{A}) \supset K_0(\widetilde{A})_+,[1])$ is computed as
\[
(\mathfrak{D} \supset \mathfrak{D}_+,1) := \varinjlim\Bigg(\tilde{\prod_{(z,\gamma)\in \widetilde{\mathfrak{Z}}_n}} \frac{\mathbb{Z}}{\dim(z)}\supset\Bigg[\tilde{\prod_{(z,\gamma)\in \widetilde{\mathfrak{Z}}_n}} \frac{\mathbb{Z}}{\dim(z)}\Bigg]_+,\mathbf{1}\Bigg)
\]
along the embeddings $\iota_{n+1,n}^{**} = \mathrm{ctr}_{n+1}\circ\iota_{n+1,n}$, $n=0,1,\dots$, where $\mathbf{1}$ is the constant function, i.e., $\mathbf{1}(z,\gamma) = 1$ for all $(z,\gamma)\in \widetilde{\mathfrak{Z}}_n$.  
\end{proposition} 

\begin{remark}\label{R5.2}{\rm For each $n$ the mapping 
\begin{equation*}
f \in \tilde{\prod_{(z,\gamma)\in \widetilde{\mathfrak{Z}}_n}} \frac{\mathbb{Z}}{\dim(z)} \mapsto
\{(z,\gamma) \mapsto \dim(z)f(z,\gamma)\} \in \mathbb{Z}^{\widetilde{\mathfrak{Z}}_n}
\end{equation*}
is an injective group homomorphism, whose image is exactly 
\begin{equation}
\big\langle\mathbb{Z}^{\widetilde{\mathfrak{Z}}_n}\big\rangle := \Bigg\{ h \in \mathbb{Z}^{\widetilde{\mathfrak{Z}}_n}\,;\, \sup_{(z,\gamma)\in\widetilde{\mathfrak{Z}}_n} \frac{|h(z,\gamma)|}{\dim(z)} <+\infty\Bigg\}. 
\end{equation}
 With these mappings $(K_0(\widetilde{A}) \supset K_0(\widetilde{A})_+,[1])$ is identified with 
 \begin{equation}
 \varinjlim \Big(\big\langle\mathbb{Z}^{\widetilde{\mathfrak{Z}}_n}\big\rangle,\big\langle\mathbb{Z}^{\widetilde{\mathfrak{Z}}_n}\big\rangle_+,\dim\Big)
 \end{equation}
 along the mapping from $\big\langle\mathbb{Z}^{\widetilde{\mathfrak{Z}}_n}\big\rangle$ to $\big\langle\mathbb{Z}^{\widetilde{\mathfrak{Z}}_{n+1}}\big\rangle$ given as the left-multiplication of $\infty\times\infty$ matrix 
 \[
\begin{bmatrix} \tilde{m}((z,\gamma),(z',\gamma'))\end{bmatrix}_{\widetilde{\mathfrak{Z}}_{n+1}\times\widetilde{\mathfrak{Z}}_n}, 
\]
where 
\[
\big\langle\mathbb{Z}^{\widetilde{\mathfrak{Z}}_n}\big\rangle_+ := 
\Big\{ h \in \big\langle\mathbb{Z}^{\widetilde{\mathfrak{Z}}_n}\big\rangle\,;\, \text{$h(z,\gamma)\geq0$ for all $(z,\gamma)\in\widetilde{\mathfrak{Z}}_n$}\Big\}
\]
 and $\dim(z,\gamma) = \dim(z)$ holds for every $(z,\gamma) \in \widetilde{\mathfrak{Z}}_n$. This description is completely consistent with dimension groups of AF-algebras. An additional feature here is that $\big\langle\mathbb{Z}^{\widetilde{\mathfrak{Z}}_n}\big\rangle$ is a much smaller set than $\mathbb{Z}^{\widetilde{\mathfrak{Z}}_n}$ except the case when $\widetilde{\mathfrak{Z}}_n$ is a finite set. }   
\end{remark}

We then investigate how the action $\widetilde{\alpha}^\gamma : \Gamma \curvearrowright \widetilde{A}$ behaves on $\mathfrak{D}$. Let $(\widetilde{\alpha}^\gamma)^*$ be the automorphism of $K_0(\widetilde{A})$ induced from $\widetilde{\alpha}^\gamma$ canonically.  

\begin{proposition} \label{P5.3} 
The automorphism $(\widetilde{\alpha}^\gamma)^{**}$ of $\mathfrak{D}$ obtained from $(\widetilde{\alpha}^\gamma)^*$ via $K_0(\widetilde{A}) \cong \mathfrak{D}$ is given as follows. For each $n \geq 0$ we have 
\[
(\widetilde{\alpha}^\gamma)^{**}(\iota_n^{**}(f)) = 
\iota_n^{**}(f\circ T_\gamma^{-1}), \qquad \gamma \in\Gamma, \quad f \in\tilde{\prod_{(z,\gamma)\in\widetilde{\mathfrak{Z}}_n}}\frac{\mathbb{Z}}{\dim(z)},
\]
where $\iota_n^{**} : \tilde{\prod}_{(z,\gamma)\in\widetilde{\mathfrak{Z}}_n}\frac{\mathbb{Z}}{\dim(z)} \to \mathfrak{D}$ is the canonical group-homomorphism. 
\end{proposition}
\begin{proof} 
Since $\widetilde{\alpha}^\gamma$ is an inductive action, the restriction of $\widetilde{\alpha}^\gamma$ to each $\widetilde{A}_n$ makes sense and induces an automorphism $(\widetilde{\alpha}^\gamma)_n^{**}$ of 
\[
\Big(K_0(\widetilde{A}_n) \overset{\Phi_n^*}{\rightarrow} K_0(\Phi_n(\widetilde{A}_n)) \overset{\mathrm{cdim}_n}{\rightarrow}\Big)\quad \tilde{\prod_{(z,\gamma)\in\widetilde{\mathfrak{Z}}_n}}\frac{\mathbb{Z}}{\dim(z)} \quad \Big(\subset \mathcal{Z}(\Phi_n(\widetilde{A}_n))\Big),
\] 
which we have to compute. This is nothing but $\mathrm{cdim}_n\circ(\widetilde{\alpha}^\gamma)^*\circ(\mathrm{cdim}_n)^{-1}$, and can be shown by the same way as abave to coincide with the restriction of $\Phi_n\circ\widetilde{\alpha}^\gamma\circ\Phi_n^{-1}$ to $
\tilde{\prod}_{(z,\gamma)\in\widetilde{\mathfrak{Z}}_n}\frac{\mathbb{Z}}{\dim(z)}$ ($\subset \mathcal{Z}(\Phi_n(\widetilde{A}_n))$). 
By \eqref{Eq4.4} one has 
\[
(\Phi_n\circ\widetilde{\alpha}^\gamma\circ\Phi_n^{-1})(z'\otimes\delta_\gamma) 
= 
z'\otimes\delta_{\gamma\gamma'}, 
\]
and hence we conclude that 
\begin{equation}\label{}
(\widetilde{\alpha}^\gamma)_n^{**}(f) = f\circ T^{-1}_\gamma, \qquad \gamma \in \Gamma, \quad f \in \tilde{\prod_{(z,\gamma)\in\widetilde{\mathfrak{Z}}_n}}\frac{\mathbb{Z}}{\dim(z)}.
\end{equation}

Since 
\begin{align*}
(\iota_{n+1,n}^{**}\circ(\widetilde{\alpha}^{\gamma''})_n^{**}(f))(z,\gamma) 
&= 
\sum_{(z',\gamma') \in \widetilde{\mathfrak{Z}}_n} \tilde{\mu}((z,\gamma),(z',\gamma'))\,f(T_{\gamma''}^{-1}(z',\gamma')) \qquad \text{(by \eqref{Eq5.3})} \\
&= 
\sum_{(z',\gamma') \in \widetilde{\mathfrak{Z}}_n} \tilde{\mu}(T_{\gamma''}^{-1}(z,\gamma),T_{\gamma''}^{-1}(z',\gamma'))\,f(T_{\gamma''}^{-1}(z',\gamma')) \quad \text{(by Proposition \ref{P4.3})} \\
&= 
\iota_{n+1,n}^{**}(f)(T_{\gamma''}^{-1}(z,\gamma)) \\
&= 
((\widetilde{\alpha}^{\gamma''})_{n+1}^{**}\circ\iota_{n+1,n}^{**}(f))(z,\gamma) 
\end{align*}
for every $(z,\gamma) \in \widetilde{\mathfrak{Z}}_{n+1}$ and $\gamma'' \in \Gamma$, the inductive limit $\varinjlim (\widetilde{\alpha}^\gamma)_n^{**}$ is well defined on $\mathfrak{D}$. Then, it is not difficult to see that this coincides with $(\widetilde{\alpha}^\gamma)^{**}$. 
\end{proof} 

Here is a proposition. 

\begin{proposition}\label{P5.4} Let $\mathcal{W}_\beta^\mathrm{ln}(\tilde{\mu})$ be all the additive maps $\psi : \mathfrak{D}_+ \to [0,\infty]$ such that 
\begin{itemize}
\item[(i)] $\psi\circ(\widetilde{\alpha}^\gamma)^{**} = \gamma^\beta\,\psi$ for all $\gamma \in \Gamma$, 
\item[(ii)] For each $n$, if $f_k \nearrow f$ in $\Big[\tilde{\prod}_{(z,\gamma)\in \widetilde{\mathfrak{Z}}_n} \frac{\mathbb{Z}}{\dim(z)}\Big]_+$ pointwisely as functions over $\widetilde{\mathfrak{Z}}_n$, then $\psi\circ\iota_n^{**}(f_k) \nearrow \psi\circ\iota_n^{**}(f)$ as $k \to \infty$. 
\item[(iii)] $\psi(\iota_0^{**}(\delta_{(1,1)})) = 1$. 
\end{itemize}
Then there is a unique affine bijection $\tilde{\nu} \in H_1^+(\tilde{\mu})_\beta \mapsto \psi_{\tilde{\nu}} \in\mathcal{W}_\beta^\mathrm{ln}(\tilde{\mu})$ so that 
\[
\psi_{\tilde{\nu}}(\iota_n^{**}(\delta_{(z,\gamma)})) = \tilde{\nu}(z,\gamma)
\]
for all $(z,\gamma) \in \widetilde{\mathfrak{Z}}_n$, $n \geq 0$. 
\end{proposition} 
\begin{proof} 
Let $\tilde{\nu} \in H_1^+(\tilde{\mu})_\beta$ be arbitrarily chosen. We observe that 
\begin{align*} 
\sum_{(z',\gamma') \in \widetilde{\mathfrak{Z}}_n} \tilde{\nu}(z',\gamma')\,f(z',\gamma') 
&= 
\sum_{(z',\gamma') \in \widetilde{\mathfrak{Z}}_n} \sum_{(z,\gamma) \in \widetilde{\mathfrak{Z}}_{n+1}} \tilde{\nu}(z,\gamma)\,\tilde{\mu}((z,\gamma),(z',\gamma'))\,f(z',\gamma') \\
&= 
\sum_{(z,\gamma) \in \widetilde{\mathfrak{Z}}_{n+1}} \tilde{\nu}(z,\gamma)\,\sum_{(z',\gamma') \in \widetilde{\mathfrak{Z}}_n} \tilde{\mu}((z,\gamma),(z',\gamma'))\,f(z',\gamma') \\
&= 
\sum_{(z,\gamma) \in \widetilde{\mathfrak{Z}}_{n+1}} \tilde{\nu}(z,\gamma)\,\iota_{n+1,n}^{**}(f)(z,\gamma)
\end{align*}
for every $f \in \Big[\tilde{\prod}_{(z,\gamma)\in \widetilde{\mathfrak{Z}}_n} \frac{\mathbb{Z}}{\dim(z)}\Big]_+$. Hence 
\[
\iota_n^{**}(f) \quad \text{with} \quad f \in \Bigg[\tilde{\prod_{(z,\gamma)\in \widetilde{\mathfrak{Z}}_n}} \frac{\mathbb{Z}}{\dim(z)}\Bigg]_+ \quad \mapsto \quad \sum_{(z,\gamma)\in\widetilde{\mathfrak{Z}}_n} \tilde{\nu}(z,\gamma)\,f(z,\gamma) 
\]
defines a well-defined additive map $\psi_\nu$ from $\mathfrak{D}_+$ to $[0,\infty]$. That the $\tilde{\nu}$ enjoys item (ii) of Definition \ref{D4.6} implies that the $\psi_{\tilde{\nu}}$ does item (i) here. That $\psi_{\tilde{\nu}}$ enjoys items (ii),(iii) is clear from its definition.  

Let $\psi \in \mathcal{W}_\beta^\mathrm{ln}(\tilde{\nu})$ be arbitrarily chosen. Define $\tilde{\nu}_\psi(z,\gamma) := \psi(\iota_n^{**}(\delta_{(z,\gamma)}))$ for each $(z,\gamma) \in \widetilde{\mathfrak{Z}}_n \subset \widetilde{\mathfrak{Z}}$. Using \eqref{Eq5.3} and item (ii) here we can easily confirm that this $\tilde{\nu}_\psi$ enjoys item (i) of Definition \ref{D4.6}. We also have, for every $(z,\gamma) \in \widetilde{\mathfrak{Z}}_n$, $n\geq0$,  
\[
\tilde{\nu}_\psi(z,\gamma) 
= 
\psi(\iota_n^{**}(\delta_{(z,\gamma)})
=
\psi(\iota_n^{**}(T_\gamma^{-1}(\delta_{(z,1)})))
=
\psi((\widetilde{\alpha}^\gamma)^{**}(\iota_n^{**}(\delta_{(z,1)})))
=
\gamma^\beta\,\psi(\iota_n^{**}(\delta_{(z,1)})) 
=
\gamma^\beta\,\tilde{\nu}_\psi(z,1),
\] 
implying that the $\tilde{\nu}_\psi$ enjoys item (ii) of Definition \ref{D4.6}. Finally, $\tilde{\nu}_\psi(1,1) = \psi(\iota_0^{**}(\delta_{(1,1)})) = 1$. Hence we are done. 
\end{proof} 

This proposition together with Theorem \ref{T4.7} gives an interpretation of $K_\beta^\mathrm{ln}(\alpha^t)$ or $H_1^+(\kappa)$ in terms of $K_0$-group. In fact, we have 

\begin{theorem}\label{T5.5} The correspondence $\omega \in K_\beta^\mathrm{ln}(\alpha^t) \mapsto \psi_\omega \in \mathcal{W}_\beta^\mathrm{ln}(\tilde{\mu})$ defined by
\[
\psi_\omega(\iota_n^{**}(\delta_{(z,\gamma)})) = \frac{\dim(z)}{\dim_\beta(z)}\,\omega(z)\,\gamma^\beta, \qquad (z,\gamma) \in \widetilde{\frak{Z}}_n, \quad n =0,1,\dots
\]
is an affine-isomorphism. In particular, each $\psi \in\mathcal{W}^\mathrm{ln}_\beta(\tilde{\mu})$ gives a unique $\omega_\psi \in K_\beta^\mathrm{ln}(\alpha^t)$ in such a way that 
\[
\omega_\psi(a) = \sum_{z\in\mathfrak{Z}_n} \psi(\iota_n^{**}(\delta_{(z,1)}))\,\frac{\mathrm{Tr}(\rho_z^{-\beta}za)}{\dim(z)}, \qquad a \in A_n, \quad n=0,1,\dots, 
\]
and any element of $K_\beta^\mathrm{ln}(\alpha^t)$ arises in this way. 
\end{theorem}

\begin{remark}\label{R5.6} {\rm Let $\mathcal{W}_\beta(K_0(\widetilde{A}))$ be all the additive maps $\psi : K_0(\widetilde{A})_+ \to [0,\infty]$ so that $\psi\circ(\alpha^\gamma)^* = \gamma^\beta\,\psi$ for all $\gamma \in \Gamma$. Then we see that $\mathcal{W}_\beta^\mathrm{ln}(\tilde{\mu})$ sits in $\mathcal{W}_\beta(K_0(\widetilde{A}))$ via $\mathfrak{D} \cong K_0(\widetilde{A})$. Note that $\mathcal{W}_\beta(K_0(\widetilde{A}))$ depends only on $\widetilde{A}$, but $\mathcal{W}_\beta^\mathrm{ln}(\tilde{\mu})$ does not.}
\end{remark}

\section{A concrete example: $\mathrm{U}_q(\infty)$}\label{S6}  

We will illustrate the present method with the infinite dimensional quantum unitary group $\mathrm{U}_q(\infty)$, for whose formulation we follow our previous paper \cite{Ueda:SIGMA22} ({\it n.b.}, the convention of $q$-deformation in both \cite{Gorin:AdvMath12},\cite{Sato:JFA19} does not fit standard references on quantum unitary group $\mathrm{U}_q(n)$, although the difference in consequences is minor, i.e., $q \rightsquigarrow q^{-1/2}$ in \cite{Gorin:AdvMath12} and $q \rightsquigarrow q^{-1}$ in \cite{Sato:JFA19}). Namely, we freely use the notations in \cite[section 4.2]{Ueda:SIGMA22}. However, Greek alphabet $\Gamma$ was used there with different meaning from this paper. 

\subsection{Weight group and Weight-extended branching system}\label{S6.1}
We first have to find the eigenvalues of $\rho_\lambda$ to determine the weight group $\Gamma$ in section \ref{S4}. Here we remark that the $\rho_\lambda$, $\lambda \in \mathbb{S}_n$, naturally satisfy $\mathrm{Tr}(\rho_\lambda) = \mathrm{Tr}(\rho_\lambda^{-1})$. 

\begin{lemma} 
The weight group $\Gamma$ is $q^\mathbb{Z} := \{q^k; k \in \mathbb{Z}\}$. 
\end{lemma}
\begin{proof}
By \cite[equation (4.17)]{Ueda:SIGMA22} we have 
\[
\rho_\lambda = \pi_\lambda(K_1^{-n+1}K_2^{-n+3}\cdots K_n^{n-1}), \qquad \lambda \in \mathbb{S}_n. 
\] 
The irreducible representations $\pi_\lambda$, $\lambda \in \mathbb{S}_n$, must satisfy that the $\pi_\lambda(K_i)$'s are commonly diagonalized with eigenvalues of the form $q^k$ including at least $q^{\lambda_i}$ for $\pi_\lambda(K_i)$. Thus, the $\rho_{(1,0)}$ ($(1,0) \in \mathbb{S}_2$) has an eigenvalue $q^{-1}$. These show $\Gamma = q^\mathbb{Z}$. 
\end{proof}

The dual of $q^\mathbb{Z}$ is identified with the $1$-dimensional torus $\mathbb{T} = \{\zeta \in \mathbb{C}; |\zeta| = 1\}$ with dual pairing $\langle q^k, \zeta\rangle = \zeta^k$ for any $k \in \mathbb{Z}$ and $\zeta \in \mathbb{T}$. The canonical surjective group-homomorphism from $\mathbb{R}$ to $\mathbb{T}$ is given by $t \mapsto q^{it}$.  

\medskip
The inductive sequence $\widetilde{W^*(\mathrm{U}_q(n))}$, $n=0,1,\dots$, is given as the $W^*$-crossed products $W^*(\mathrm{U}_q(n))\,\bar{\rtimes}_{\vartheta_n^\zeta}\mathbb{T}$. By Proposition \ref{P4.3} its branching graph is given by $\bigsqcup_{n\geq0} \mathbb{S}_n\times q^\mathbb{Z}$ and the multiplicity function is computed by finding the spectral decomposition $\rho_\lambda\rho_{\lambda'}^{-1}$ on $\mathcal{H}_\lambda \overset{\pi_\lambda}{\curvearrowleft} U_q\mathfrak{gl}(n+1)$ with $(\lambda,\lambda') \in \mathbb{S}_{n+1}\times\mathbb{S}_n$, $\lambda' \prec \lambda$, $n \geq 0$.  
As in \cite[section 4.4.5]{Ueda:SIGMA22} we obtain 
\begin{equation} 
\rho_\lambda\rho_{\lambda'}^{-1} 
= 
\pi_{\lambda'}(K_1^{-1}\cdots K_n^{-1})\otimes\pi_{(|\lambda|-|\lambda'|)}(K_1^n) 
\end{equation}
(up to unitary equivalence), where the right-hand side is the representation of $U_q\mathfrak{gl}(n)\otimes U_q\mathfrak{gl}(1)$. Since the branching rule from $U_q\mathfrak{gl}(n) \hookrightarrow U_q\mathfrak{gl}(n+1)$ is the same as the classical case and hence multiplicity-free, we obtain that $\rho_\lambda\rho_{\lambda'}^{-1}$ is of the form $\gamma\,z_\lambda z_{\lambda'}$ with positive scalar $\gamma > 0$ and also that 
\begin{align*} 
\mathrm{Tr}(z_\lambda z_{\lambda'}) 
&= 
s_{\lambda'}(1,\dots,1)\,s_{(|\lambda|-|\lambda'|)}(1) = s_{\lambda'}(1,\dots,1) = \dim(\lambda'), \\
\mathrm{Tr}(\rho_\lambda\rho_{\lambda'}^{-1}) 
&= 
s_{\lambda'}(q^{-1},\dots,q^{-1})\,s_{(|\lambda|-|\lambda'|)}(q^n) \\
%&= 
%q^{n|\lambda|-(n+1)|\lambda'|} s_{\lambda'}(1,\dots,1)\,s_{(|\lambda|-|\lambda'|)}(1) \\
&= 
q^{n|\lambda|-(n+1)|\lambda'|} s_{\lambda'}(1,\dots,1) = q^{n|\lambda|-(n+1)|\lambda'|}\dim(\lambda'). 
\end{align*} 
It follows that $\gamma = q^{n|\lambda|-(n+1)|\lambda'|}$ and hence
\begin{equation}
q_{(\lambda,\lambda')}(q^k) = 
\begin{cases}
z_\lambda z_{\lambda'} & (k = n|\lambda|-(n+1)|\lambda'|), \\
0 & \text{(otherwise)}. 
\end{cases}
\end{equation}
Therefore, we obtain
\begin{equation} 
\tilde{m}((\lambda,q^k),(\lambda',q^\ell)) 
= 
\begin{cases}
1 & (\ell-k = n|\lambda|-(n+1)|\lambda'|), \\
0 & \text{(otherwise)} 
\end{cases}
\end{equation}
({\it n.b}., $n|\lambda|-(n+1)|\lambda'| = [\lambda', (|\lambda|-|\lambda'|)]$; see \cite[section 4.4.5]{Ueda:SIGMA22} for this terminology).  
Hence we have determined the branching graph of the $\widetilde{W^*(\mathrm{U}_q(n))}$, $n=0,1,\dots$, completely. With \cite[equation (4.16)]{Ueda:SIGMA22} we remark that this computation is consistent with the construction in section \ref{S4.3}. This is not surprise, because this computation as well as the computation of link \cite[equation (4.16)]{Ueda:SIGMA22} were done by using only the branching rule. 

\subsection{Quantum group interpretation of weight-extension}\label{S6.2} 
We clarify that the algebra $\widetilde{W^*(\mathrm{U}_q(n))} = W^*(\mathrm{U}_q(n))\,\bar{\rtimes}_{\vartheta_n^\zeta}\mathbb{T}$ comes from a compact quantum group. A similar (but not the same) algebra was appeared in an unpublished manuscript of De Commer \cite{DeCommer:preprint}, where $q^\mathbb{Z}$ is replaced with $q^{2\mathbb{Z}}$.  

\medskip
Let $(\mathbb{C}[\mathbb{T}],\Delta_\mathbb{T},S_\mathbb{T},\varepsilon_\mathbb{T})$ be the Hopf $*$-algebra associated with the $1$-dimensional torus $\mathbb{T}$, that is, $\mathbb{C}[\mathbb{T}]$ denotes all the Laurent polynomials $\sum_k c_k \chi_k$ ($c_k \in \mathbb{C}$) in the continuous functions $C(\mathbb{T})$ with $\chi_k(\zeta) = \zeta^k$ in $\zeta \in \mathbb{T}$ ($k \in \mathbb{Z}$), and 
\[
\Delta_\mathbb{T}(\chi_k) = \chi_k\otimes\chi_k, \qquad S_\mathbb{T}(\chi_k) = \chi_{-k}, \qquad \varepsilon_\mathbb{T}(\chi_k) = 1. 
\]
Since $S_\mathbb{T}^2 = \mathrm{id}$, the Woronowicz character or the special positive element of $\mathcal{U}(\mathbb{T})$, the algebraic dual of $\mathbb{C}[\mathbb{T}]$, must be trivial by \cite[Proposition 1.7.9]{NeshveyevTuset:Book13}. 

We define the new Hopf $*$-algebra $(\mathbb{C}[\mathrm{U}_q(n)\times\mathbb{T}], \Delta_{n,\mathbb{T}}, S_{n,\mathbb{T}}, \varepsilon_{n,\mathbb{T}})$ to be $\mathbb{C}[\mathrm{U}_q(n)\times\mathbb{T}] := \mathbb{C}[\mathrm{U}_q(n)]\otimes\mathbb{C}[\mathbb{T}]$ (algebraic tensor product) and 
\begin{equation*} 
\Delta_{n,\mathbb{T}} := \Sigma_{23}\circ(\Delta_n\otimes\Delta_\mathbb{T}), \qquad 
S_{n,\mathbb{T}} := S_n\otimes S_\mathbb{T}, \qquad 
\varepsilon_{n,\mathbb{T}} := \varepsilon_n\otimes\varepsilon_\mathbb{T},  
\end{equation*} 
where $\Sigma$ is the tensor-flip map and we used the leg-notation. The matrix elements of unitary representations $U\otimes\chi_k$ with finite dimensional unitary representations $U$ of $(\mathbb{C}[\mathrm{U}_q(n)],\Delta_n)$ and $k\in\mathbb{Z}$ clearly generate $\mathbb{C}[\mathrm{U}_q(n)\times\mathbb{T}]$ as algebra, and hence the Hopf $*$-algebra indeed defines a compact quantum group by \cite[Theorem 1.6.7]{NeshveyevTuset:Book13}. The corresponding $C^*$-algebra is trivially $C(\mathrm{U}_q(n))\otimes C(\mathbb{T})$ with unique $C^*$-tensor product due to nuclearity. Moreover, the unitary irreducible representations $U_\lambda\otimes\chi_k$, $(\lambda,k)\in\mathbb{S}_n\times\mathbb{Z}$, are easily shown to be mutually inequivalent, and we can prove that they form a complete family of inequivalent, unitary irreducible representations by appealing to the famous orthogonal relation and Peter--Weyl type theorem (see \cite[Theorem 1.4.3(ii) and the discussion following Corollary 1.5.5]{NeshveyevTuset:Book13}). Consequently, 
\begin{equation}
\mathcal{U}(\mathrm{U}_q(n)\times\mathbb{T}) = \prod_{(\lambda,m)\in\mathbb{S}_n\times\mathbb{Z}} B(\mathcal{H}_{(\lambda,m)}) \qquad \text{with}\quad \mathcal{H}_{(\lambda,m)}:=\mathcal{H}_\lambda,
\end{equation}
and hence 
\begin{equation}\label{Eq6.5}
W^*(\mathrm{U}_q(n)\times\mathbb{T}) = \bigoplus_{(\lambda,m)\in\mathbb{S}_n\times\mathbb{Z}} B(\mathcal{H}_{(\lambda,m)}) = W^*(\mathrm{U}_q(n))\,\bar{\otimes}\,\ell^\infty(\mathbb{Z}), 
\end{equation}
which is clearly isomorphic to $\widetilde{W^*(\mathrm{U}_q(n))}$ via $\Phi_n$ of Lemma \ref{L4.1}. 

Choose an $x \in W^*(\mathrm{U}_q(n)) \subset \mathcal{U}(\mathrm{U}_q(n))$ and a $\zeta \in \mathbb{T}$. We regard $\zeta$ as an element of $\mathcal{U}(\mathbb{T})$ by $\zeta(f) := f(\zeta)$ for every $f \in C(\mathbb{T})$, in which $\mathbb{C}[\mathbb{T}]$ sits. For any $a,b \in \mathbb{C}[\mathrm{U}_q(n)]$ and $k,\ell \in \mathbb{Z}$ we have  
\begin{align*} 
(\hat{\Delta}_{n,\mathbb{T}}(x\otimes\zeta))((a\otimes\chi_k)\otimes(b\otimes\chi_\ell)) 
&= 
(x\otimes\zeta)(ab\otimes\chi_{k+\ell}) \\
&=
x(ab) \zeta^{k+\ell} \\
&= 
\hat{\Delta}_n(x)(a\otimes b)\,\hat{\Delta}_\mathbb{T}(\chi_k\otimes\chi_\ell) \\
&= 
(\hat{\Delta}_n(x)_{13}\hat{\Delta}_\mathbb{T}(\zeta)_{24})((a\otimes\chi_k)\otimes(b\otimes\chi_\ell)), 
\end{align*} 
and hence
\begin{equation} 
\hat{\Delta}_{n,\mathbb{T}}(x\otimes\zeta) 
= 
\hat{\Delta}_n(x)_{13}\hat{\Delta}_\mathbb{T}(\zeta)_{24} 
= 
\hat{\Delta}_n(x)_{13}(1\otimes\zeta\otimes1\otimes\zeta) 
\in 
\mathcal{U}((\mathrm{U}_q(n)\times\mathbb{T})^2). 
\end{equation} 
We observe that 
\[
\zeta = \sum_{k\in\mathbb{Z}} \zeta(\chi_k)\,\delta_k = \sum_{k\in\mathbb{Z}} \zeta^k\,\delta_k \in \ell^\infty(\mathbb{Z}) \subset \mathbb{C}^\mathbb{Z} = \mathcal{U}(\mathbb{T}).
\]
Since $\hat{\Delta}_n(x) \in W^*(\mathrm{U}_q(n))\,\bar{\otimes}\,W^*(\mathrm{U}_q(n))$ and the $\zeta \in \mathbb{T}$ generate $\ell^\infty(\mathbb{Z})$ as $W^*$-algebra, we conclude that the restriction of $\hat{\Delta}_{n,\mathbb{T}}$ to $W^*(\mathrm{U}_q(n)\times\mathbb{T})$ coincides with the injective normal $*$-homomorphism
\[
\Sigma_{23}\circ(\hat{\Delta}_n\,\bar{\otimes}\,\hat{\Delta}_\mathbb{T}) : 
W^*(\mathrm{U}_q(n)\times\mathbb{T}) \to (W^*(\mathrm{U}_q(n)\times\mathbb{T}))^{\bar{\otimes}2}.  
\]
It is also easy to see that the restrictions of $\hat{\varepsilon}_{n,\mathbb{T}}$ and $\hat{S}_{n,\mathbb{T}}$ to $\mathcal{U}(\mathrm{U}_q(n))\otimes\mathcal{U}(\mathbb{T})$ (sitting in $\mathcal{U}(\mathrm{U}_q(n)\times\mathbb{T})$ naturally) are exactly $\hat{\varepsilon}_n\otimes\hat{\varepsilon}_\mathbb{T}$ and $\hat{S}_n\otimes\hat{S}_\mathbb{T}$, respectively. In particular, the restriction of $\hat{\varepsilon}_{n,\mathbb{T}}$ to $W^*(\mathrm{U}_q(n)\times\mathbb{T})$ is $\hat{\varepsilon}_n\,\bar{\otimes}\,\hat{\varepsilon}_\mathbb{T}$. Since $\mathcal{F}(\mathrm{U}_q(n)\times\mathbb{T}) = \mathcal{F}(\mathrm{U}_q(n))\otimes c_\mathrm{fin}(\mathbb{Z})$ (algebraic tensor product) with all the finitely supported bi-sequences $c_\mathrm{fin}(\mathbb{Z})$, we have $\hat{S}_{n,\mathbb{T}}^2 = \hat{S}_n^2\otimes\mathrm{id}$ on $\mathcal{F}(\mathrm{U}_q(n)\times\mathbb{T})$. 

We observe that $(U_\lambda\otimes\chi_k)^\mathrm{cc} = U_\lambda^\mathrm{cc}\otimes\chi_k$ by definition and hence $\rho_{(\lambda,k)} = \rho_\lambda\otimes1$ by \cite[Proposition 1.4.4]{NeshveyevTuset:Book13} for every $(\lambda,k) \in \mathbb{S}_n\times\mathbb{Z}$. Therefore, the special positive element for $\mathrm{U}_q(n)\times\mathbb{T}$ must be $\rho_n\otimes1 \in \mathcal{U}(\mathrm{U}_q(n))\otimes\mathbb{C}1 \subset \mathcal{U}(\mathrm{U}_q(n)\times\mathbb{T})$. It follows that the restriction of the unitary antipode $\hat{R}_{n,\mathbb{T}}$ to $W^*(\mathrm{U}_q(n)\times\mathbb{T})$ coincides with $\hat{R}_n\,\bar{\otimes}\,\hat{S}_\mathbb{T}$.  
 
\medskip
Regarding the $\Phi_n$ in Lemma \ref{L4.1} as a map from $\widetilde{W^*(\mathrm{U}_q(n))}$ onto $W^*(\mathrm{U}_q(n)\times\mathbb{T}) = W^*(\mathrm{U}_q(n))\,\bar{\otimes}\,\ell^\infty(\mathbb{Z})$ (see \eqref{Eq6.5}) we observe that 
\begin{equation}
\Phi_n(\pi_{\vartheta_n}(x)) = x\otimes1, \qquad \Phi_n(\lambda(q^{it})) = \rho_n^{it}\otimes q^{it}, \qquad x \in W^*(\mathrm{U}_q(n)), \quad  t\in\mathbb{R}. 
\end{equation}
Hence, via $\Phi_n$, the Hopf $*$-algebra structure $(\hat{\Delta}_{n,\mathbb{T}},\hat{R}_{n,\mathbb{T}},\vartheta_{n,\mathbb{T}}^t = \mathrm{Ad}(\rho_n^{it}\otimes1),\varepsilon_{n,\mathbb{T}})$ on $W^*(\mathrm{U}_q(n)\times\mathbb{T})$ is transferred to that on $\widetilde{W^*(\mathrm{U}_q(n))}$ as follows. Write $\tilde{\Delta}_n := (\Phi_n^{\bar{\otimes}2})^{-1}\circ\hat{\Delta}_{n,\mathbb{T}}\circ\Phi_n$, $\tilde{R}_n := \Phi_n^{-1}\circ\hat{R}_{n,\mathbb{T}}\circ\Phi_n$, $\tilde{\vartheta}_n^t := \Phi_n^{-1}\circ\vartheta_{n,\mathbb{T}}^t\circ\Phi_n$ ({\it n.b.,} this does not correspond to $\widetilde{\alpha}_n^\gamma$ in section \ref{S3}) and $\tilde{\varepsilon}_n^t := \hat{\varepsilon}_{n,\mathbb{T}}\circ\Phi_n$ for simplicity. Then we have
\begin{align*} 
\tilde{\Delta}_n(\pi_{\vartheta_n}(x)\lambda(q^{it})) 
&= 
\pi_{\vartheta_n}^{\bar{\otimes}2}(\hat{\Delta}_n(x))\,(\lambda(q^{it})\otimes\lambda(q^{it})),  \\
\tilde{R}_n(\pi_{\vartheta}(x)\lambda(q^{it}))  
&= 
\lambda(q^{-it})\,\pi_{\vartheta_n}(\hat{R}_n(x)), \\
\tilde{\vartheta}_n^t(\pi_{\vartheta_n}(x)\lambda(q^{it})) 
&= 
\pi_{\vartheta_n}(\vartheta_n^t(x))\lambda(q^{it})), \\
\tilde{\varepsilon}_n(\pi_{\vartheta_n}(x)\lambda(q^{it})) 
&= \hat{\varepsilon}_n(x)
\end{align*}
for any $x \in W^*(\mathrm{U}_q(n))$ and $t \in \mathbb{R}$. Thus, $\widetilde{W^*(\mathrm{U}_q(n))}$ is equipped with the natural structure of the group $W^*$-algebra of compact quantum group $\mathrm{U}_q(n)\times\mathbb{T}$. 

It is easy to see that the dual action of $q^k \in q^\mathbb{Z}$ acts on an generator $x\otimes\delta_m \in W^*(\mathrm{U}_q(n))\,\bar{\otimes}\,\ell^\infty(\mathbb{Z}) = W^*(\mathrm{U}_q(n)\times\mathbb{T})$ as $x\otimes\delta_m \mapsto x\otimes\delta_{m+k}$.  

\medskip
So far, we have seen that each $\widetilde{W^*(\mathrm{U}_q(n))}$ becomes a `compact quantum group'. Moreover, the above computations show that the resulting quantum group structure is compatible with the embedding $\widetilde{W^*(\mathrm{U}_q(n))} \hookrightarrow \widetilde{W^*(\mathrm{U}_q(n+1))}$, $n\geq0$. The embedding is  interpreted, on the $W^*(\mathrm{U}_q(n)\times\mathbb{T})$, $n=0,1,\dots$, as 
\begin{equation}
\begin{aligned} 
x\otimes1 = \Phi_n(\pi_{\vartheta_n}(x)) 
&\mapsto %&\longleftrightarrow 
\Phi_{n+1}(\pi_{\vartheta_{n+1}}(x)) = x\otimes1 \qquad (x \in W^*(\mathrm{U}_q(n))), \\
\rho_n^{it}\otimes q^{it} = \Phi_n(\lambda(q^{it})) 
&\mapsto %&\longleftrightarrow 
\Phi_{n+1}(\lambda(q^{it})) = \rho_{n+1}^{it}\otimes q^{it} \qquad (t \in \mathbb{R}), 
\end{aligned}
\end{equation}
or other words,  
\begin{equation} 
x\otimes q^{it} \mapsto (x(\rho_n^{-1}\rho_{n+1})^{it})\otimes q^{it} \qquad (x \in W^*(\mathrm{U}_q(n)),\ t \in \mathbb{R}). 
\end{equation} 
Here, we remark, see \cite[section 4.2.1]{Ueda:SIGMA22}, that 
\[
(\rho_n^{-1}\rho_{n+1})^{it} = (\rho_{n+1}\rho_n^{-1})^{it} = \rho_n^{-it}\rho_{n+1}^{it} = \rho_{n+1}^{it}\rho_n^{-it} \in W^*(\mathrm{U}_q(n))' \cap W^*(\mathrm{U}_q(n+1))
\] 
for every $t \in \mathbb{R}$. Namely, the choice of embedding of $\mathrm{U}_q(n)\times\mathbb{T} \hookrightarrow \mathrm{U}_q(n+1)\times\mathbb{T}$ is not standard. Thus, \emph{although the $\rho_n$, $n=0,1,\dots$, do NOT form an inductive sequence in any sense, the $\rho_n^{it}\otimes q^{it}$, $n=0,1,\dots$, do thanks to the weight-extension of $\mathcal{B}(\mathrm{U}_q(\infty)) = \varinjlim W^*(\mathrm{U}_q(n))$.} This became possible by the famous Fell absorption principle! 

Finally, the projection $e_{q^k}(n)$ in $L(\mathbb{T}) := \lambda(\mathbb{T})'' \subset \widetilde{W^*(\mathrm{U}_q(n))}$ becomes 
\begin{equation} 
e_{q^k}(n) := \sum_{\ell \in \mathbb{Z}} \sum_{\lambda \in \mathbb{S}_n} p_\lambda(q^{k-\ell})\otimes \delta_\ell 
\end{equation} 
in $W^*(\mathrm{U}_q(n)\times\mathbb{T})$, where the double sums can be  interchanged and $\rho_\lambda = \sum_{k \in \mathbb{Z}} q^k\,p_\lambda(q^k)$ (a finite sum; {\it n.b.}, all but finitely many $p_\lambda(q^k)=0$) is the spectral decomposition as in section \ref{S4}. In fact, for any $x \in z_\lambda W^*(\mathrm{U}_q(n))$ we have 
\[
e_1(n)(x\otimes 1)e_1(n) = \sum_{\ell \in \mathbb{Z}} (p_\lambda(q^{-\ell}) x p_\lambda(q^{-\ell}))\otimes\delta_\ell, 
\]   
and 
\[
\mathrm{ctr}_n(e_1(n)(x\otimes 1)e_1(n)) = \sum_{\ell\in\mathbb{Z}} \frac{\mathrm{Tr}(p_\lambda(q^{-\ell})\,x)}{\dim(\lambda)}(1\otimes\delta_\ell). 
\]
Hence, if we assign $q^{-\ell}$ at $1\otimes\delta_\ell$, then the above element becomes $\mathrm{Tr}(\rho_\lambda x)/\dim(\lambda)$. This is a closer look at trick behind Theorem \ref{T3.7} in the quantum group setting. 

\begin{remark}\label{R6.2} {\rm The discussion in this subsection is completely general. Actually, the same interpretation in terms of quantum groups is applicable to any inductive sequence of compact quantum groups, where the $1$-dimensional torus $\mathbb{T}$ and its dual $\mathbb{Z} = \widehat{\mathbb{T}}$ in the above should be replaced with the dual $G$ of the weight group $\Gamma$ and $\Gamma$ itself, respectively. 
}
\end{remark}

}

\end{document}